\documentclass[a4paper,11pt,english,leqno]{article} 

\usepackage[T1]{fontenc}
\usepackage[utf8]{inputenx}
\usepackage{amsthm,amsmath,amssymb}
\usepackage{graphicx}
\usepackage{geometry}
\usepackage{hyperref}
\usepackage{xcolor}
\usepackage{comment}
\usepackage{dsfont}
\usepackage{enumerate}  
\usepackage{enumitem}
\usepackage{stmaryrd}
\usepackage{float}
\usepackage{tikz}
\usepackage{comment}
\usepackage{thmtools}
\usepackage{cleveref}
\usepackage{biblatex}
\usepackage{doi}
\usepackage{caption}
\usepackage{subcaption}
\newcommand{\N}{\mathbb N}	
\newcommand{\Fr}{\mathrm{Fr}}

\newcommand{\RMS}{\mathrm{RMS}}
\newcommand{\CPS}{\mathrm{CPS}}
\newcommand{\CP}{\mathrm{CP}}
\newcommand{\RM}{\mathrm{RM}}
\newcommand{\Z}{\mathbb Z}	
	
\newcommand{\R}{\mathbb R}	
\newcommand{\E}{\mathbb E}

\newcommand{\Card}{\mathrm{Card}}	
\newcommand{\Var}{\mathrm{Var}}	

\renewcommand{\P}{\mathbb P}	
\renewcommand{\L}{\mathcal L}

\newcommand{\eps}{\varepsilon}

\theoremstyle{plain}
\newtheorem{theorem}{Theorem}[section]
\newtheorem{lemma}[theorem]{Lemma}

\newtheorem{proposition}[theorem]{Proposition}

\theoremstyle{definition}

\newtheorem{definition}{Definition}[section]

\title{Richardson’s model and the contact
process with stirring: long time behavior.}

\date{}

\begin{document}

\begin{flushleft}
\LARGE \textbf{Richardson's model and the contact process with stirring: long time behavior}\\
\vspace{0,5cm}
\large \textbf{Régine Marchand, Irène Marcovici and Pierrick Siest}
\end{flushleft}

\begin{flushleft}

Université de Lorraine, IECL, 54000 Nancy\\
Université de Rouen Normandie, LMRS, 76801 Saint-Étienne-du-Rouvray\\
E-mail address: \texttt{regine.marchand@univ-lorraine.fr, irene.marcovici@univ-rouen.fr, pierrick.siest@univ-lorraine.fr}\\
\vspace{0,5cm}

\end{flushleft}

\noindent \textbf{Abstract.} We study two famous interacting particle systems, the so-called Richardson's model and the contact process, when we add a stirring dynamics to them. We prove that they both satisfy an asymptotic shape theorem, as their analogues without stirring, but only for high enough infection rates, using couplings and restart techniques. We also show that for Richardson's model with stirring, for high enough infection rates, each site is forever infected after a certain time almost surely. Finally, we study weak and strong survival for both models on a homogeneous infinite tree, and show that there are two phase transitions for certain values of the parameters and the dimension, which is a result similar to what is proved for the contact process.

\section{Introduction}

In 1973, Richardson introduced the so-called \emph{Richardson's model} \cite{Richardson1973}, which can be seen as the simplest interacting particle system on $\Z^d$ to model the evolution over time of the spread of an epidemic: each infected site infects its neighbors at a fixed rate $\lambda>0$. In 1974, Harris \cite{Harris1974} introduced the \emph{contact process}, which is obtained by adding a healing dynamics to Richardson's model: each infected site heals at rate $1$. The contact process is a non-permanent model, that is, the epidemic can die out. A natural question is then: for which infection rates $\lambda$ has the epidemic a positive probability to spread forever? Harris \cite{Harris1974} showed that, in all dimensions $d$, the contact process exhibits a phase transition. This means that there exists a constant $\lambda_c^{\CP}(d)\in (0,+\infty)$, called the \emph{critical parameter}, such that the probability, starting from a single infected point, that the epidemic spreads forever is continuous in $\lambda$, and positive if and only if $\lambda>\lambda_c(d)$. 

The aim of this paper is to obtain various results on Richardson's model and the contact process, when we add a stirring dynamics to them: two neighboring sites exchange their states at a fixed rate $\nu>0$. When these models are considered to model the propagation of an epidemic, stirring represents the movements of individuals. Durrett and Neuhauser \cite{DurrettNeuhauser1994} initiated the study of this kind of models in the nineties: they worked on the phase transition of some interacting particle systems in which the particles are stirred at a fast rate. Katori \cite{Katori1994} and Konno \cite{Konno1994} obtained a more precise description of the asymptotic behavior of the critical parameter $\lambda_c(\nu)$, seen as a function of the stirring rate $\nu$. More recently, some improvements were made on the same question, see Berezin and Mytnik \cite{BerezinMytnik2014}, Levit and Valesin \cite{LevitValesin2017}, and Mytnik and Shlomov \cite{MytnikShlomov2021}. All these results are about the asymptotic behavior of the critical parameter $\lambda_c(\nu)$. Here we are interested in growth properties of Richardson's model and the contact process with stirring, as detailed below. 

For Richardson's model, as well as for the contact process when the epidemic spreads forever, determining the asymptotic behavior of the set of the (once) infected sites is a natural question. One of the most famous results for this type of question was obtained by Cox and Durrett \cite{CoxDurrett1981} for first-passage percolation on $\Z^d$: they proved an \emph{asymptotic shape theorem} for the set of sites reached before time $t$, when $t$ goes to infinity. 
Asymptotic shape theorems are proved, by Richardson \cite{Richardson1973} for Richardson's model, and by Durrett--Griffeath \cite{DurrettGriffeath1982} and Bezuidenhout--Grimmett \cite{BezuidenhoutGrimmett1990} for the contact process, for the set of (once) infected sites, using ergodic subadditivity theory. Since then, a lot of variations of the contact process were studied. Some examples are the \emph{contact process in random environment}, introduced by Bramson, Durrett and Schonmann \cite{BramsonDurrettSchonmann1991}, the \emph{boundary modified contact process}, (Durrett and Schinazi \cite{DurrettSchinazi2000}), the \emph{contact process in random environment} (Garet and Marchand \cite{GaretMarchand2012}), the \emph{contact process with aging} (Deshayes \cite{Aurélia2014}), or more recently the \emph{contact process in an evolving random environment} (Seiler and Sturm \cite{Seiler2023}). For each of these models, an asymptotic shape theorem is proved or conjectured. Here we prove an asymptotic shape theorem for Richardson's model and contact process with stirring, for large enough infection rates. We prove it by showing that these two models satisfy some linear growth properties. This allows us to use Deshayes and Siest's \cite{DeshayesSiest2024} result: they proved that a class of random linear growth models, which also contains some of the other models mentioned in this paragraph, satisfies an asymptotic shape theorem. 

We also worked on the question of \emph{fixation}: does a site stay infected forever after a certain time? It is not obvious for Richardson's model with stirring, since an infected site can be healed if it exchanges its state with a healthy neighbor. But if an infected site is surrounded by other infected sites, it is possible that it cannot get healed by an exchange after a certain time, once the epidemic developed sufficiently. We prove that for large enough infection rates, each site fixes on the infected state forever, after a certain time.

Concerning the contact process with stirring, a site cannot be infected forever, since it recovers at rate $1$, independently of the rest of the configuration. However, it can be healthy forever after a certain time, even if the epidemic survives. This question is linked to the notions of \emph{weak survival} and \emph{strong survival}. We say that there is weak survival if the process survives with positive probability, and that there is strong survival if the set of times $t\in \R_+$ such that the origin is infected at time $t$ is infinite, with positive probability. Since the state $\emptyset$ is absorbing, strong survival clearly implies weak survival. If there is weak survival but not strong survival, then almost surely each site is healthy forever after a certain time, despite positive probability of the presence of infected sites at any time: the set of (actually) infected sites "goes away" from any finite box over time. 

For the contact process on $\Z^d$, it is known (see \cite{LiggettIPS}) that once the epidemic has a positive probability to spread forever, there is strong survival. On a homogeneous tree, that is, an infinite tree in which all the vertices have the same degree, the behavior of the contact process is very different: Pemantle \cite{Pemantle1992} and Liggett \cite{Liggett1996} proved that for certain infection rates, there is weak survival, but not strong survival, for the contact process on a homogeneous tree of degree $d\ge 3$. Note that Stacey \cite{Stacey1996} also proved that there are two phase transitions for all degrees $d\ge 3$, but he did not use bounds on critical parameters in his proof. We prove a similar result for Richardson's model and contact process with stirring, strongly inspired by Liggett's techniques (see \cite{LiggettSIS}), and we obtain explicit bounds for the critical parameter of the strong survival for both the RMS and the CPS. 

In Section $2$, we define Richardson's model with stirring and the contact process with stirring. In Section $3$, we state the results we prove for these two models. In Sections $4$ to $6$, we prove our results. In Section $7$, we discuss our results and propose some open questions.

\section{Definitions of the models} \label{Section Définition du modèle}

All the processes that will be studied or used in this paper are nearest neighbor interacting particle systems on the graph $(\Z^d,\E^d)$, where $d\in \N^*$ and $\E^d:=\{(x,y)\in (\Z^{d})^2:\sum_{i=1}^d |x_i-y_i|=1\}$ is the set of oriented edges between nearest neighbors of $\Z^d$, or on a homogeneous tree $T_d$ of degree $d+1$. We give only the definitions on $(\Z^d,\E^d)$: they translate straightforwardly on $T_d$. The set $\{0,1\}^{\Z^d}$ is called the set of \emph{configurations}: vertices (which we call \emph{sites} in the rest of the paper) with value $1$ are called \emph{infected}, and the others are called \emph{healthy}.
Since the map
$$\begin{array}{cccc}
& \{0,1\}^{\Z^d} & \to & \Z^d \\
 & \omega & \mapsto & \{x\in \Z^d ~:~ \omega(x)=1\} \\
\end{array}$$
is bijective, then we can see a configuration as a subset of $\Z^d$ containing the infected sites. 

For $x,y\in \Z^d$ and $\eta \in \{0,1\}^{\Z^d}$, we denote by:
\begin{itemize}
    \item $\eta_{x}$ the configuration identical to $\eta$, except for $\eta(x)$ which is replaced by $1-\eta(x)$ (the state of site $x$ is flipped).
    \item $\eta_{x,y}$ the configuration identical to $\eta$, except for $\eta(x)$ and $\eta(y)$, which are replaced respectively by $\eta(y)$ and $\eta(x)$ (we exchange the states of sites $x$ and $y$). Note that $\eta_{x,y}=\eta_{y,x}$.
\end{itemize}

Let $\lambda,\gamma$ and $\nu$ be positive real numbers. Three interactions can occur in our models: the infection of a healthy site, the healing of an infected site, and the stirring, which corresponds to the exchange of states of two neighboring sites. We thus define three different infinitesimal generators, associated with each interaction.

We start by the generator $\mathcal{L}^{I}$, which describes the infections, and appears in all the models. We denote by $\mathcal{C}_0(\{0,1\}^{\Z^d})$ the set of continuous functions from $\{0,1\}^{\Z^d}$ to $\R$ which depend only on finitely many coordinates of the configuration. A site becomes infected at a rate equal to the number of its infected neighbors. For all $f\in \mathcal{C}_0(\{0,1\}^{\Z^d})$ and every configuration $\eta\in \{0,1\}^{\Z^d}$, we set:
\begin{align*}
    \mathcal{L}^{I} f(\eta) &=\sum_{x\in\Z^d} \left[\mathds{1}_{\{\eta(x)=0\}} \sum_{y\sim x} \eta(y) \right]\left[f(\eta_{x})-f(\eta)\right].
\end{align*}
The second generator $\mathcal{L}^{H}$ describes the healing dynamics, and appears in the contact processes. A site becomes healthy at rate $1$. For all $f\in \mathcal{C}_0(\{0,1\}^{\Z^d})$ and every configuration $\eta\in \{0,1\}^{\Z^d}$, we set:
\begin{align*}
    \mathcal{L}^{H} f(\eta) &=\sum_{x\in\Z^d}\mathds{1}_{\{\eta(x)=1\}} \left[f(\eta_{x})-f(\eta)\right].
\end{align*}
The third generator $\mathcal{L}^{S}$ describes the stirring dynamics. Two neighboring sites exchange their states at rate $1$. Note that the exchange modifies the configuration only if one site is infected and the other is healthy. For all $f\in \mathcal{C}_0(\{0,1\}^{\Z^d})$ and every configuration $\eta\in \{0,1\}^{\Z^d}$, we set:
\begin{align}
    \mathcal{L}^{S} f(\eta) &=\sum_{e=(x,y)\in\E^d}\mathds{1}_{\{\eta(x)=1,\eta(y)=0\}}\left[f(\eta_{x,y})-f(\eta)\right] \notag\\
    &=\sum_{e=\{x,y\}\in\E^d}\mathds{1}_{\{\eta(x)\ne\eta(y)\}}\left[f(\eta_{x,y})-f(\eta)\right]. \label{Générateur du mélange}
\end{align}

Now we define the models that will be used in this paper. The first one is \emph{Richardson's model} with \emph{infection rate} $\lambda$ (denoted by $\RM(\lambda)$). In this model, there are only infections. Richardson's model is a Markov process with generator $\mathcal{L}^{\RM}_\lambda$, given by:
\begin{align*}
    \mathcal{L}^{\RM}_\lambda =\lambda\mathcal{L}^{I}.
\end{align*}
The \emph{contact process} with infection rate $\lambda$ and \emph{healing rate} $\gamma$ (denoted by $\CP(\lambda,\gamma)$) is an extension of Richardson's model, to which we add a healing dynamics. The contact process is a Markov process with generator $\mathcal{L}^{\CP}_{\lambda,\gamma}$, given by:
\begin{align*}
    \mathcal{L}^{\CP}_{\lambda,\gamma}=\lambda\mathcal{L}^{I}+\gamma\mathcal{L}^{H}.
\end{align*}
\emph{Richardson's model with stirring} with infection rate $\lambda$ and \emph{stirring rate} $\nu$ (denoted by $\RMS(\lambda,\nu)$) is an extension of Richardson's model, to which we add a stirring dynamics. Richardson's model with stirring is a Markov process with generator $\mathcal{L}^{\RMS}_{\lambda,\nu}$, given by:
\begin{align}
    \mathcal{L}^{\RMS}_{\lambda,\nu} &=\lambda\mathcal{L}^{I}+\nu\mathcal{L}^{S}.\label{Définitions des modèles: générateur du RMS}
\end{align}
The \emph{contact process with stirring} with infection rate $\lambda$, healing rate $\gamma$, and stirring rate $\nu$ (denoted by $\CPS(\lambda,\gamma,\nu)$), is an extension of the contact process, to which we add a stirring dynamics. We can also see it as an extension of Richardson's model with stirring, to which we add a healing dynamics. All the three interactions can happen: infection, healing and stirring. The contact process with stirring is a Markov process with generator $\mathcal{L}^{\CPS}_{\lambda,\gamma,\nu}$, given by:
\begin{align}
    \mathcal{L}^{\CPS}_{\lambda,\gamma,\nu}=\lambda\mathcal{L}^{I}+\gamma\mathcal{L}^{H}+\nu\mathcal{L}^{S}.
    \label{Définitions des modèles: générateur du CPS}
\end{align}

Richardson's model with stirring and the contact process with stirring are the two models on which we obtain new results. Note that, by a rescaling of time, we can restrict ourselves to the study of the $\RMS(\lambda,1)$ and the $\CPS(\lambda,1,\nu)$. Note also that the existence of these processes comes from Theorem B3 of Liggett \cite{LiggettSIS}. They are all Feller Markov processes, therefore they verify the strong Markov property.




\section{Results} \label{Section Résultats}

\subsection{Asymptotic shape theorem}
We will see in Subsection \ref{Section Construction graphique} that we can define our models with a graphical construction. As a direct consequence of this construction, we will obtain the following properties:
\begin{lemma} \label{Définition du modèle: le processus est un growth model}
Let $(\xi_t)_{t\ge 0}$ be a $\RMS$ or a $\CPS$.
\begin{enumerate}
    \item The law of the process is invariant by the translation $T_x$, for all $x\in \Z^d$. 
    \item The process is \emph{additive}, that is: there exists a coupling such that, for all subsets $A\subset B$ of $\Z^d$, the processes $(\xi_t^{A})_{t\ge 0}$ and $(\xi_t^{B})_{t\ge 0}$ verify, for all $t\geq 0$,
$$\xi_t^{A\cup B}= \xi_t^{A}\cup \xi_t^{B}.$$ 
    \item The empty set is an absorbing state, that is: if $\xi_{t_0}=\emptyset$, then for all $t>t_0$, we have $\xi_{t}=\emptyset$.
\end{enumerate}
\end{lemma}
We are interested in the asymptotic behavior of the set of once infected sites: Figure~\ref{fig:Simulations du RMS et CPS} shows some simulations of the RMS and the CPS. We define, for all $x,y\in \Z^d$, the \emph{lifetime of the process} $(\xi_t^x)_{t\ge 0}$:
$$\tau^x=\inf\{t\ge 0 ~:~ \xi_t^x=\emptyset\},$$
and the \emph{hitting time of the site} $y$ by the process $(\xi_t^x)_{t\ge 0}$:
$$t^x(y)=\inf\{t\ge 0 ~:~ y\in\xi_t^x\}.$$
When $x=0$, we simply write $\tau$ and $t(y)$. When $\P(\tau=+\infty)>0$, we say that the process \emph{survives with positive probability}. Durrett \cite{Durrett1980} proved in 1980 a one-dimensional version of an asymptotic shape theorem for this set, in the supercritical contact process. His result also holds for a larger class of models, which he calls growth models. By Lemma \ref{Définition du modèle: le processus est un growth model}, both the RMS and the CPS are growth models in the sense of Durrett. Therefore, in dimension one we directly obtain the following asymptotic shape theorems, in the entire supercritical region:

\begin{theorem} \label{Théorème TFA d=1}
Suppose that $d=1$. We set $r_t=\sup\{x\in \Z:x\in \xi_t\}$ and $l_t=\inf\{x\in \Z:x\in \xi_t\}$.
\begin{enumerate}
    \item Let $\lambda>0$ and $(\xi_t)_{t\ge 0}$ be a $\RMS(\lambda,1)$. There exists $\alpha>0$ such that:
    $$\lim_{t\to +\infty} \frac{r_t}{t}=-\lim_{t\to +\infty} \frac{l_t}{t}=\alpha\quad\text{a.s.}$$
    \item Let $\nu>0$, $\lambda>\lambda_c^{\CPS}(1,\nu)$ and $(\xi_t)_{t\ge 0}$ be a $\CPS(\lambda,1,\nu)$. There exists $\alpha>0$ such that:
    $$\lim_{t\to +\infty} \frac{r_t}{t}=-\lim_{t\to +\infty} \frac{l_t}{t}=\alpha\quad\text{a.s. on }\{\tau=+\infty\}.$$
\end{enumerate}
\end{theorem}

\begin{figure}[H]
\captionsetup[subfigure]{labelformat=empty}
    \centering
    \begin{subfigure}{0.45\textwidth}
        \centering
        \includegraphics[width=4.5cm,height=4.5cm]{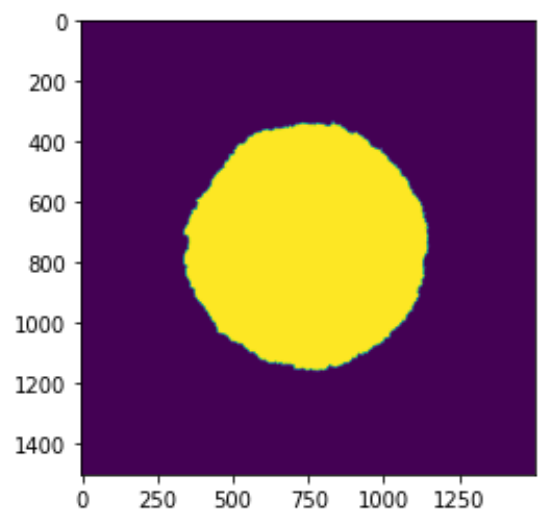}
        \caption{$\RMS(1,1)$.}
    \end{subfigure}
    \hspace{0.5cm}
    \begin{subfigure}{0.45\textwidth}
        \centering
        \includegraphics[width=4.5cm,height=4.5cm]{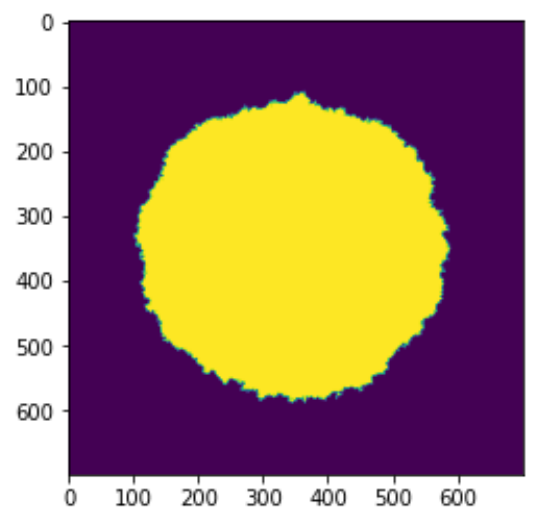}
        \caption{$\RMS(0.2,1)$.}
    \end{subfigure}
    
    \vspace{0.5cm}
    
    \begin{subfigure}{0.45\textwidth}
        \centering
        \includegraphics[width=4.5cm,height=4.5cm]{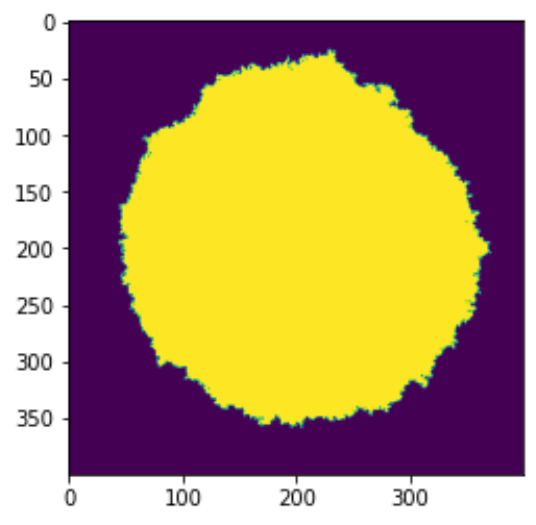}
        \caption{$\CPS(2,1,10)$.}
    \end{subfigure}
    \hspace{0.5cm}
    \begin{subfigure}{0.45\textwidth}
        \centering
        \includegraphics[width=4.5cm,height=4.5cm]{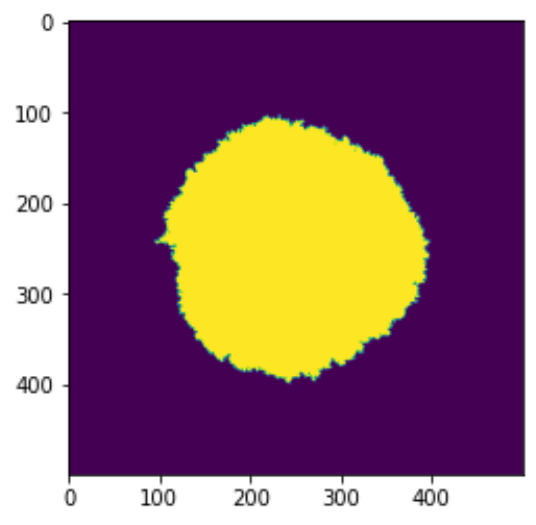}
        \caption{$\CPS(2,1,1)$.}
    \end{subfigure}
\caption{Simulations of the set of once infected sites for some Richardson's models with stirring and contact processes with stirring, after 1 000 000 interactions (infections/healings/exchanges).}
\label{fig:Simulations du RMS et CPS}
\end{figure}

In higher dimensions, a shape theorem is harder to obtain. We want to apply Deshayes and Siest's \cite{DeshayesSiest2024} result: they proved an asymptotic shape theorem for a class of growth models with linear growth properties. They defined two classes of Markov processes: the first one, the class $\mathcal{C}$, is a class of what they called growth models, and is similar to the class of growth models Durrett \cite{Durrett1980} defined. By Lemma \ref{Définition du modèle: le processus est un growth model}, both the RMS and the CPS are in class $\mathcal{C}$. The second class, the class $\mathcal{C}_L$, is a class of models which have linear growth properties.
In our context of invariance by spatial translations of the law of the RMS and the CPS, these two models are in class $\mathcal{C}_L$ if there exist $C_1,C_2,M_1,M_2>0$ such that for all $t>0$ and $x\in\Z^d$,
\begin{align}
\P(\tau=+\infty)&>0, \label{Le modèle est permanent} \\
\P\left(\exists y\in\Z^d: t(y)\leq t\text{ and }\|y\|\ge M_1 t\right)&\leq C_1\exp(-C_2t), \label{AML} \\
\P\left(t<\tau<+\infty\right)&\leq C_1\exp(-C_2 t),\label{SC} \\
\P\left(t(x)\geq M_2\|x\|+t,\tau=+\infty\right)&\leq C_1\exp(-C_2t),\label{ALL} 
 \end{align}
where $||x||=\sum_{i=1}^n |x_i|$. Property \eqref{Le modèle est permanent} means that the process survives with positive probability, when it starts from the initial configuration $\xi_0=\{0\}$. Property \eqref{AML} (resp. \eqref{ALL}) corresponds to at most linear growth (resp. at least linear growth) of the set of once infected sites. Property \eqref{SC} is a "small cluster" property, by analogy with percolation vocabulary: if the epidemic dies out, then it dies out quickly. Therefore, the epidemic starting from a singleton in the graphical construction is either a small finite connected component, or an infinite connected component, with high probability. 

Proving these inequalities is the hardest part to prove the asymptotic shape theorem: we are only able to show them for large enough infection rates. We will do a coupling between the RMS, a RM and a CP that is restarted when it dies out (we denote it by $\overline{\CP}$) in such a way that 
$$\overline{\CP}\subset \RMS \subset \RM,$$
seeing the models as the set of their infected sites, in order to prove \eqref{Le modèle est permanent}, \eqref{AML}, \eqref{SC} and \eqref{ALL} for the RMS. We do the same coupling to prove these inequalities for the CPS (only the parameters of the CP and the RM change). These couplings will be given in Lemma \ref{Lemme_sur_le_couplage}. It gives sufficient information only if the CP associated to the $\overline{\CP}$ is supercritical, which is the case only for infection rates $\lambda$ sufficiently large. In this case, we can apply Theorem $1$ of Deshayes and Siest \cite{DeshayesSiest2024} to obtain an asymptotic shape theorem. This approach is similar to Durrett and Griffeath's approach for the contact process \cite{DurrettGriffeath1982}:  they did a coupling between a CP in dimension $d$ and a CP in dimension $1$ to prove the linear growth of the former, using linear growth properties they proved for the supercritical CP in dimension $1$ in \cite{DurrettGriffeath1983}. With this coupling, they obtained an asymptotic shape theorem for the contact process in all dimensions, but only for infection rates $\lambda>\lambda_c^{\CP}(1)$. The extension to the whole supercritical region was done by Bezuidenhout and Grimmett \cite{BezuidenhoutGrimmett1990} in 1990. 

For a norm $\mu$ on $\R^d$, we denote by $B_{\mu}$ the unit ball for this norm. For all $t\ge 0$, we set
$$H_t=\{x\in \Z^d:t(x)\le t\}+[0,1)^d=\{x\in \Z^d:\exists s<t,~x\in \xi_s\}+[0,1)^d.$$
We obtain the following asymptotic shape theorems in dimension $d\ge 2$:

\begin{theorem} \label{Théorème TFA d>1}
Suppose that $d\ge 2$.
\begin{enumerate}
    \item Let $\lambda>2d\lambda_c^{\CP}(d)$ and $(\xi_t)_{t\ge 0}$ be a $\RMS(\lambda,1)$. There exists a norm $\mu$ on $\R^d$ such that, for all $\eps>0$,
    $$\P(\exists T\in \R^+,~ \forall t\geq T ~:~ (1-\eps)tB_{\mu} \subset H_t \subset (1+\eps)tB_{\mu})=1.$$ 
    \item Let $\nu>0$, $\lambda>(2d\nu+1)\lambda_c^{\CP}(d)$ and $(\xi_t)_{t\ge 0}$ be a $\CPS(\lambda,1,\nu)$. There exists a norm $\mu$ on $\R^d$ such that, for all $\eps>0$,
    $$\P(\exists T\in \R^+,~ \forall t\geq T ~:~ (1-\eps)tB_{\mu} \subset H_t \subset (1+\eps)tB_{\mu}~|~\tau=+\infty)=1.$$ 
\end{enumerate}
\end{theorem}

In their article \cite{DurrettGriffeath1982}, Durrett and Griffeath mentioned that their asymptotic shape theorem can be extended to a class of linear growth models they defined. Note that if we prove \eqref{Le modèle est permanent}, \eqref{AML}, \eqref{SC} and \eqref{ALL} for the RMS or the CPS, then in addition to Lemma \ref{Définition du modèle: le processus est un growth model}, both the RMS and the CPS are linear growth models in the sense of Durrett and Griffeath, and so the asymptotic shape theorem can also be proved using their result \cite{DurrettGriffeath1982}. 

Our final growth result is about the asymptotic behavior of the number of infected sites at time $t$ in the RMS. We wanted to use this result to prove the at least linear growth of the model for all infection rates $\lambda>0$, but we did not manage to. We discuss this result in Section \ref{Section discussion}.

\begin{proposition} \label{Théorème Proposition sur densité de sites infectés}
    Let $\lambda>0$ and $(\xi_t)_{t\ge 0}$ be a $\RMS(\lambda,1)$. There exist constants $A=A_{d}$ and $B=B_{C,\lambda,d}>0$ such that for all $t$ large enough:
    $$\P[\Card(\xi_t)\ge B t^d]\ge 1-\frac{A}{\lambda t}.$$
\end{proposition}

\subsection{Fixation of the sites in the RMS}

Among the models we study here, the RMS has a particularity: there is no direct healing as in the contact processes, but an infected site can heal if it exchanges its state with a healthy neighbor. It is a big difference since an infected site has to have healthy neighbors to heal, whereas in the contact processes a site can always heal, no matter the neighboring. Therefore, the question of fixation of sites naturally arises:  does each site stay infected forever after a sufficiently long time, like in Richardson's model? We prove a result for all infection rates in dimension $1$ and $2$, and for $\lambda$ large enough in higher dimensions. In high dimensions with low infection rates, we have a weaker result.

We say that a site $x\in \Z^d$ \emph{fixes on the infected state} (resp. \emph{fixes on the healthy state}) if:
\begin{equation*}
     \P(\exists T>0,~\forall t\geq T,~ x\in \xi_t)=1 
\end{equation*}
\begin{equation*}
    \text{(resp. }\P(\exists T>0,~ \forall t\geq T,~ x\notin \xi_t)=1). 
\end{equation*}

We prove the following result:

\begin{theorem} \label{Théorème Fixation RMS}
Let $\lambda>0$ and $(\xi_t)_{t\ge 0}$ be $\RMS(\lambda,1)$ starting from a non-empty initial configuration.
\begin{enumerate}
    \item All sites fix on the infected state or all sites fix on the healthy state.
    \item Suppose that $d\in \{1,2\}$. Then all sites fix on the infected state.
    \item Suppose that $d\ge 3$ and $\lambda>2d\lambda_c(d)$. Then all sites fix on the infected state.
\end{enumerate}
\end{theorem}


\subsection{Weak survival and strong survival on a homogeneous tree}

Here we study two different notions of survival for the contact process starting from the initial configuration $\{0\}$. The first one, which we call \emph{weak survival}, is what we call "to survive with positive probability" in the other sections: the process $(\xi_t)_{t\ge 0}$ survives \emph{weakly} if: 
$$\P(\forall t\ge 0,~\xi_t\ne \emptyset)>0.$$
We also define a stronger version, which imposes that the set of infected sites does not "go away" from any finite box. We say that the process survives \emph{strongly} if: 
$$\P(\forall T\ge 0,~ \exists t>T,~ 0\in \xi_t)>0.$$
Since the configuration $\emptyset$ is absorbing, strong survival implies weak survival. 

On $\Z^d$, when they survive weakly, both RM and CP also survive strongly: there is only one phase transition. It is obvious for RM, but for CP it is not easy to prove: it uses Bezuidenhout and Grimmett's block construction \cite{BezuidenhoutGrimmett1990}, a proof can be found in Liggett's book \cite{LiggettSIS}. On an infinite homogeneous tree $T_d$, that is, an infinite tree having with all vertices having degree $d+1$, Pemantle \cite{Pemantle1992} proved that the contact process exhibits two phase transitions: there exist infection rates $\lambda$ such that $\CP(\lambda,1)$ survives weakly, but not strongly. Note that $\RMS(\lambda,1)$ survives weakly for all $\lambda>0$, since $(\Card(\xi_t))_{t\ge 0}$ is non-decreasing for this model, but there is no reason that it survives strongly for all $\lambda>0$. We prove that there are also two phase transitions for the RMS and for the CPS, for a set of couples $(d,\nu)$ of dimension and stirring rate, on $T_d$:

\begin{theorem} \label{Théorème Survie forte et faible sur Td}
\begin{enumerate}
    \item Let $\lambda>0$ and $(\xi_t)_{t\ge 0}$ be $\RMS(\lambda,1)$ on the homogeneous tree $T_d$. There exists $\lambda_s^{\RMS}(d)>0$ such that:
    \begin{itemize}
        \item if $\lambda>\lambda_s^{\RMS}(d)$, then the $\RMS(\lambda,1)$ survives strongly,
        \item if $0<\lambda<\lambda_s^{\RMS}(d)$, then the $\RMS(\lambda,1)$ survives weakly but does not survive strongly.
    \end{itemize}
    Moreover, we have:
    \begin{align}
        \frac{d+1}{2\sqrt{d}}-1\le\lambda_s^{\RMS}(d)\le\frac{d+1}{\sqrt{d}-1}. \label{Bornes paramètre critique survie forte RMS}
    \end{align} 
    \item Let $\lambda,\nu>0$ and $(\xi_t)_{t\ge 0}$ be a $\CPS(\lambda,1,\nu)$. There exists $\lambda_s^{\CPS}(d,\nu)>0$ such that:
    \begin{itemize}
        \item if $\lambda>\lambda_s^{\CPS}(d,\nu)$, then the $\CPS(\lambda,1,\nu)$ survives strongly,
        \item if $0<\lambda<\lambda_s^{\CPS}(d,\nu)$, then the $\CPS(\lambda,1,\nu)$ does not survive strongly,
    \end{itemize}
    and we have 
    \begin{align}
        \nu\left(\frac{d+1}{2\sqrt{d}}-1\right)\le\lambda_s^{\CPS}(d,\nu)\le \frac{(d+1)\nu+1}{\sqrt{d}-1}.\label{Bornes paramètre critique survie forte CPS}
    \end{align}
    Moreover:
    \begin{itemize}
        \item For all $(d,\nu)$ in 
    $$W=\left\{(d,\nu)\in \llbracket 2,+\infty\llbracket \times (0,+\infty)~:~\nu\left(\frac{d+1}{2\sqrt{d}}-1\right)-\frac{(d+1)\nu+1}{d-1}> 0\right\},$$
    there exist infection rates $\lambda>0$ such that $\CPS(\lambda,1,\nu)$ survives weakly, but not strongly.
        \item The set $W$ is non-empty if and only if $d\ge 17$.
    \end{itemize}
    \end{enumerate}
\end{theorem}

In the next sections, we prove the results of Section \ref{Section Résultats}.

\section{Proof of the asymptotic shape theorem \ref{Théorème TFA d>1}} \label{Section Preuve du TFA}

\subsection{Graphical construction} \label{Section Construction graphique} 
Both Richardson's model and the contact process have a graphical construction with Poisson point processes on $\R^+$. It can be used to obtain a graphical construction, which allows us to use percolation techniques, see Harris \cite{Harris1978}. We extend this kind of construction to the $\RMS(\lambda,1)$ and the $\CPS(\lambda,1,\nu)$.

We endow $\R^+$ with the Borel $\sigma$-algebra $\mathfrak{B}(\R^+)$, and we denote by $M$ the set of locally finite counting measures $m=\sum_{i=1}^{+\infty} \delta_{t_i}$ on $\R^+$, which can be seen as sequences of jump times. We endow $M$ with the $\sigma$-algebra $\mathcal{M}$ generated by the maps $m\mapsto m(B)$, with $B\in \mathfrak{B}(\R^+)$. We define the measurable space $(\Omega,\mathcal{F})$ by 
\begin{align}(\Omega,\mathcal{F})=(M^{\E^d}\times M^{\E^d}\times M^{\Z^d},\mathcal{M}^{\otimes \E^d}\otimes \mathcal{M}^{\otimes \E^d}\otimes \mathcal{M}^{\otimes \Z^d}). \label{Définition de l'espace mesurable}
\end{align}
Remember that $\E^d$ is the set of \emph{oriented} edges between nearest neighbors of $\Z^d$. On this measurable space, we define the probability measures
$$\P_{\RMS(\lambda,1)}=\mathcal{N}_{\lambda}^{\otimes \E^d}\otimes \mathcal{N}_{1}^{\otimes \E^d}\otimes \delta_{\emptyset}^{\Z^d} \quad\text{and}\quad\P_{\CPS(\lambda,1,\nu)}=\mathcal{N}_{\lambda}^{\otimes \E^d}\otimes \mathcal{N}_{1}^{\otimes \E^d}\otimes\mathcal{N}_{\nu}^{\otimes \Z^d},$$ 
where for all $\mu>0$, $\mathcal{N}_{\mu}$ is the law of a Poisson process of intensity $\mu$, and $\delta_{\emptyset}$ is the counting measure associated with an empty set of jump times. The subscripts in $\P_{\RMS(\lambda,1)}$ and $\P_{\CPS(\lambda,1,\nu)}$ will often be removed when the choice of the model we study and the values of the parameters are clear.

We extend Harris' graphical construction \cite{Harris1978} to our models: we provide an informal description of it, see Figure~\ref{Representation_graphique} to visualize. We only present the case of the CPS, since the construction for the RMS is the same, except that there is no healing. Let $\omega=((\omega^I_e)_{e\in \E^d},(\omega^S_e)_{e\in \E^d},(\omega^H_z)_{z\in \Z^d})\in \Omega$. We draw a time line $\R^+$ above each site. On top of each site $z$, we draw squares at times given by $\omega^H_z$, which correspond to healing times. On top of each edge $e$, we draw horizontal arrows between the two endpoints of $e$, with the same orientation as the edge $e$, at times given by $\omega^I_e$ and $\omega^S_e$. Each of these arrows has a mark: $I$ for a time given by $\omega^I_e$ (infection time), and $S$ for a time given by $\omega^S_e$ (stirring time). 

We define \emph{open paths} on a configuration. An open path follows the time lines, using the horizontal arrows to jump from one time line to another, with the following constraints:
\begin{itemize}
    \item the path cannot go through a square,
    \item the path can follow a horizontal arrow with a I mark (that is, an infection time) when it encounters one, but it can also ignore it and stay on the same time line,
    \item when the path encounters a horizontal arrow with a S mark (that is, a stirring time), then it is forced to follow it if the endpoint of the arrow is healthy at this time, and cannot follow it if the endpoint is infected.
\end{itemize}
Note that the first two constraints are exactly those of an open path in the graphical construction of a contact process. Now, for any $A\subset \Z^d$, we define a process $(\xi^A_t)_{t\ge 0}$ on $\mathcal{P}(\Z^d)$ as follows: we fix $\xi_0=A$, and for all $t>0$, $y\in \xi_t^A$ if and only if there exists $x\in A$ such that there is an open path from $(x,0)$ to $(y,t)$. A direct extension of Harris's result \cite{Harris1978} proves that for any $A\subset \Z^d$, the process $(\xi_t^A)_{t\ge 0}$ is a well-defined Markov process with a generator $\mathcal{L}^{\CPS}_{\lambda,1,\nu}$, see \eqref{Définitions des modèles: générateur du CPS}, and initial configuration $\xi_0^A=A$.
Note that our exchanges are oriented: this will be convenient for coupling models with stirring and models without stirring.

Finally, to prove the additivity property (see Lemma \ref{Définition du modèle: le processus est un growth model}) for our models, we modify this graphical construction by changing the constraint for the stirring arrows: these arrows become non-oriented edges, and a path can always follow a stirring edge. The process obtained with this construction has the law of the same contact process: see the definition of the generator \eqref{Générateur du mélange}.
\begin{figure}
    \centering
    \includegraphics[width=6.5cm, height=5.5cm]{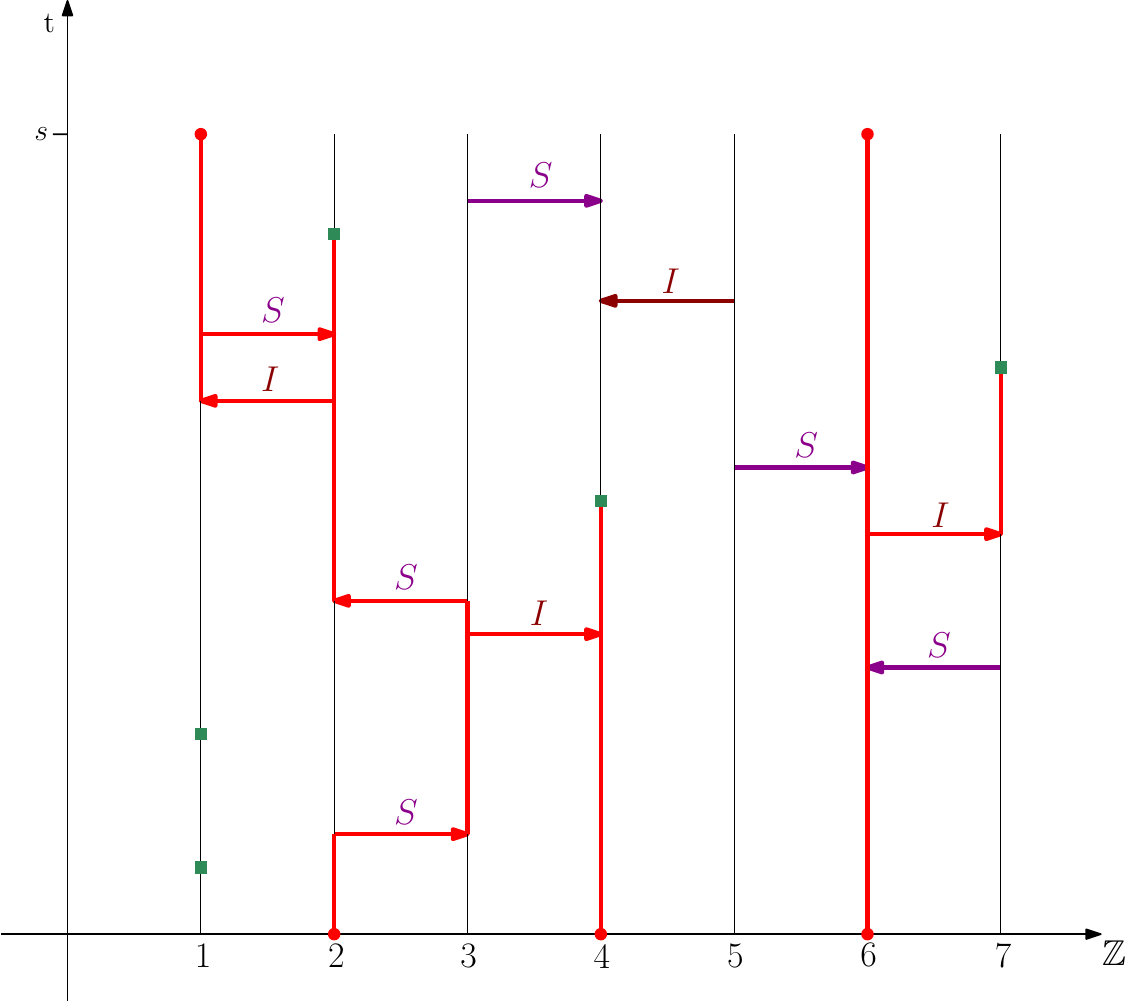}
    \caption{Graphical construction for the CPS on $\Z$, with initial configuration $A=\{2,4,6\}$, from time $0$ to time $s$. Brown (resp. purple) arrows with I-marks (resp. S-marks) correspond to infection times (resp. stirring time). Green squares correspond to healing times. Open paths starting from an infected site are in red. The set $\xi_s^A$ of infected sites at time $s$, starting from the initial configuration $A$, is the set of intersections between the open paths and the horizontal line of equation $t=s$: here it is equal to $\{1,6\}$. Note that the constraint of following a stirring arrow if and only if the site at the endpoint is healthy is necessary: site $1$ would be healthy at time $s$ if passing by a stirring arrow was mandatory, and site $3$ would be infected at time $s$ if it was optional.}
    \label{Representation_graphique}
\end{figure}

\subsection{Critical parameter of the CPS}

Harris \cite{Harris1974} proved that, for all dimensions $d\ge 1$, there exists a constant $\lambda_c^{\CP}(d)$, called the \emph{critical parameter of the contact process}, such that: 
\begin{itemize}
    \item for all $\lambda>\lambda_{c}^{\CP}(d)$, $\CP(\lambda,1)$ survives with positive probability,
    \item for all $\lambda<\lambda_{c}^{\CP}(d)$, $\CP(\lambda,1)$ dies out almost surely.
\end{itemize}

There is also a transition phase for the CPS. For all $\nu>0$, we define the \emph{survival function} $\theta_{\nu}$ of the CPS (at fixed stirring rate $\nu$, and healing rate $1$) as the map
$$\begin{array}{cccc}
\theta_{\nu}:&(0,+\infty) & \longrightarrow &  [0,1]\\
&\lambda & \longmapsto & \P_{\CPS(\lambda,1,\nu)}(\tau=+\infty)
\end{array}.$$
It follows directly from the graphical construction that $\theta_{\nu}$ is a non-decreasing map. This means that, for all dimension $d\ge 1$ and stirring rate $\nu$, there exists a constant $\lambda_{c}^{\CPS}(d,\nu)\in [0,+\infty]$, called the \emph{critical parameter} for the $\CPS$ of stirring rate $\nu$ in dimension $d$, such that:
\begin{itemize}
    \item for all $\lambda>\lambda_{c}^{\CPS}(d,\nu)$, $\P_{\CPS(\lambda,1,\nu)}(\tau=+\infty)>0$,
    \item for all $\lambda<\lambda_{c}^{\CPS}(d,\nu)$, $\P_{\CPS(\lambda,1,\nu)}(\tau=+\infty)=0$.
\end{itemize}

A comparison with a branching process gives $\lambda_{c}^{\CPS}(d,\nu)>0$. We can do a coupling between the CPS and a CP to prove that $\lambda_{c}^{\CPS}(d,\nu)<+\infty$: we will do it in Section \ref{Section Preuve du TFA}, see Lemma \ref{Lemme_sur_le_couplage}. We do not know the behavior of the CPS at the critical point $\lambda_{c}^{\CPS}(d,\nu)$.

\subsection{Translation operators, lifetime and hitting time} \label{Definition_du_modèle:definitions des_translations_de_base}

Translations in time and space are naturally defined on $\Omega$. For all $t\geq 0$, we define the translation operator on a locally finite counting measure $m=\sum_{i=1}^{+\infty} \delta_{t_i}$ on $\R^+$ by
$$\theta_t (m)= \sum\limits_{i=1}^{+\infty} \mathds{1}_{\{t_i\geq t\}} \delta_{t_i-t}.$$
The translation $\theta_t$ induces an operator on $\Omega$, still denoted by $\theta_t$: for all $\omega\in \Omega$, we set
$$\theta_t (\omega) = (\theta_t (\omega_e))_{e\in \E^d}.$$
This translation is a translation in time in the following sense: for all $t,t_0\geq 0$, $\theta_t(\xi_{t_0}^A)$ is the set of points $(y,t_0+t)$, with $y\in \Z^d$, such that there exists an open path in the graphical representation between $(x,t)$ and $(y,t_0+t)$, where $x\in A$. For all $x\in \Z^d$, we also define the following operator:
$$\forall \omega\in \Omega,~ T_x(\omega):= (\omega_{x+e})_{e\in \E^d},$$
where $x+e$ is the edge $e$ translated by the vector $x$. It is a translation in space in the following sense: for all $x\in \Z^d$ and $t\geq 0$, $T_x(\xi_t^A)=\xi_t^{A_x}$, where $A_x$ is the translation of the set $A\subset \Z^d$ by the vector $x$. Note that, since Poisson processes are invariant under translation and both $\P_\lambda$ and $\P_{\lambda,\nu}$ are product measures, these measures are stationary under $\theta_t$ and under $T_x$, for all $t\geq 0$ and $x\in \Z^d$.

In this section, we prove Theorem \ref{Théorème TFA d>1}. We want to prove inequalities \eqref{Le modèle est permanent}, \eqref{AML}, \eqref{SC} and \eqref{ALL} in order to apply Theorem 1 of Deshayes and Siest \cite{DeshayesSiest2024}. 

For all $\lambda>0$, the process $\RMS(\lambda,1)$ survives: we have $\P_\lambda(\tau=+\infty)=1$, so \eqref{Le modèle est permanent} is verified. For the CPS with fixed stirring rate $\nu$, we have to suppose that $\lambda>\lambda_c^{\CPS}(d,\nu)$ in order to have \eqref{Le modèle est permanent}. 

For the other inequalities, we use couplings between our models, contact processes and Richardson's models. We start with an informal description, using the graphical construction. We do it first for the $\CPS(\lambda,1,\nu)$. We are going to construct two processes, $(\eta_t)_{t\ge 0}$ and $(\zeta_t)_{t\ge 0}$.

The process $(\eta_t)_{t\ge 0}$ is obtained in the following way: each stirring time on the edge $(x,y)$ is replaced by a healing square at the same time at $x$, and the rest of the construction stays the same. See Figure~\ref{fig:Couplage_CPS_et_CP} to see the graphical representation of these two processes. There are only infection arrows and healing squares for the process $(\eta_t)_{t\ge 0}$: it is a contact process. Since the infections are the same as for the CPS, then the infection rate is $\lambda$. Each site has $2d$ neighbors and the stirring rate of the $\CPS(\lambda,1,\nu)$ is $\nu$, therefore the healing rate is $2d\nu+1$. 

The process $(\zeta_t)_{t\ge 0}$ is obtained in the following way: each stirring arrow is replaced by an infection arrow, and the rest of the construction stays the same. See Figure~\ref{fig:Couplage_CPS_et_RM}. There are only infection arrows in the graphical construction of the process $(\zeta_t)_{t\ge 0}$: it is Richardson's model, with infection rate $\lambda+\nu$.

For the $\RMS(\lambda,1)$, since its graphical construction is the same as that of the $\CPS(\lambda,1,\nu)$ minus the healing squares, we can proceed in the same way to define a process $(\zeta_t)_{t\ge 0}$, which is a $\RM(\lambda+1)$, and a process $(\eta_t)_{t\ge 0}$, which is a $\CP(\lambda,2d)$. 

We do a formal construction of our couplings in the following lemma.

\begin{figure}
    \begin{minipage}{0.49\linewidth}
    \raggedright
    \includegraphics[height=4cm, width=5.5cm]{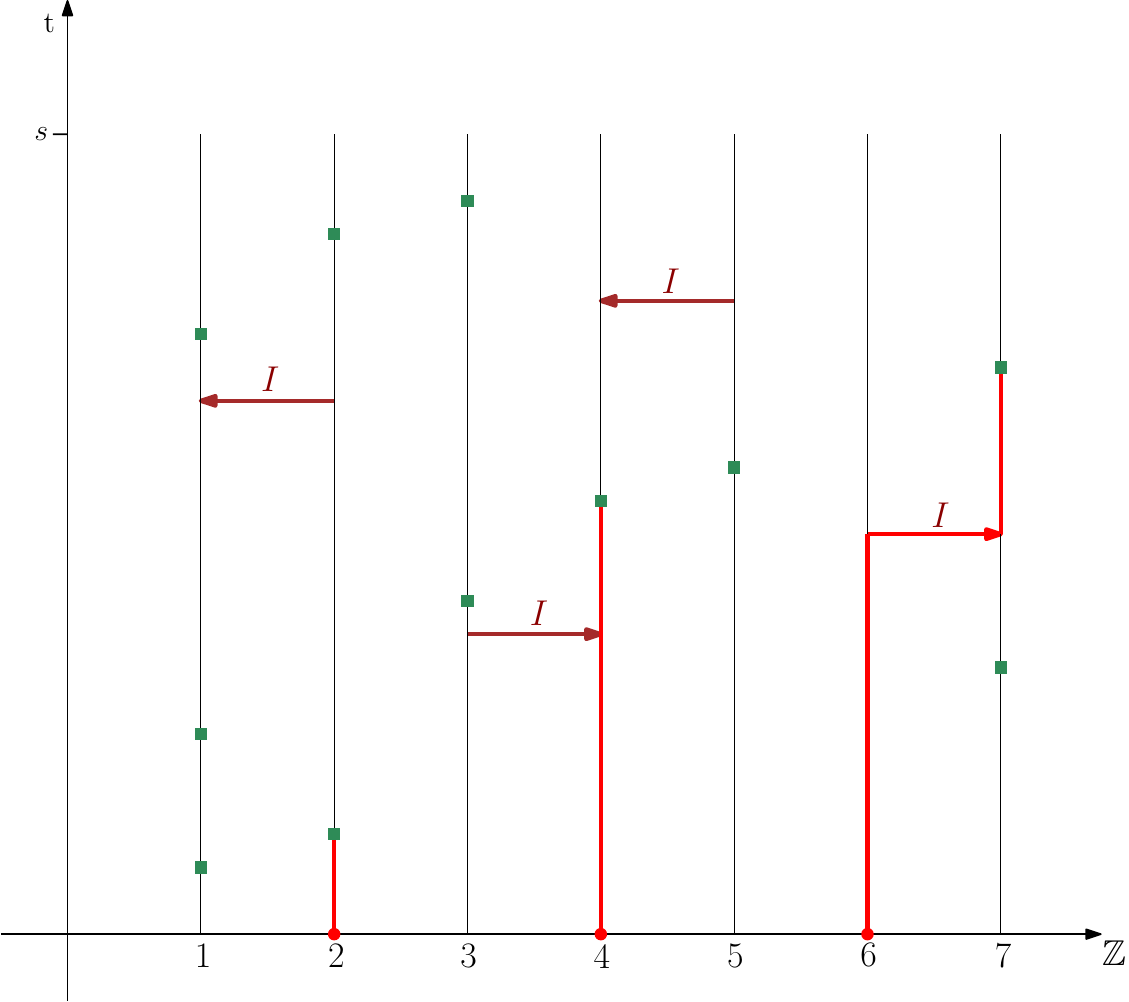}
    \end{minipage}
    \begin{minipage}{0.49\linewidth}
    \raggedleft
    \includegraphics[height=4cm, width=5.5cm]{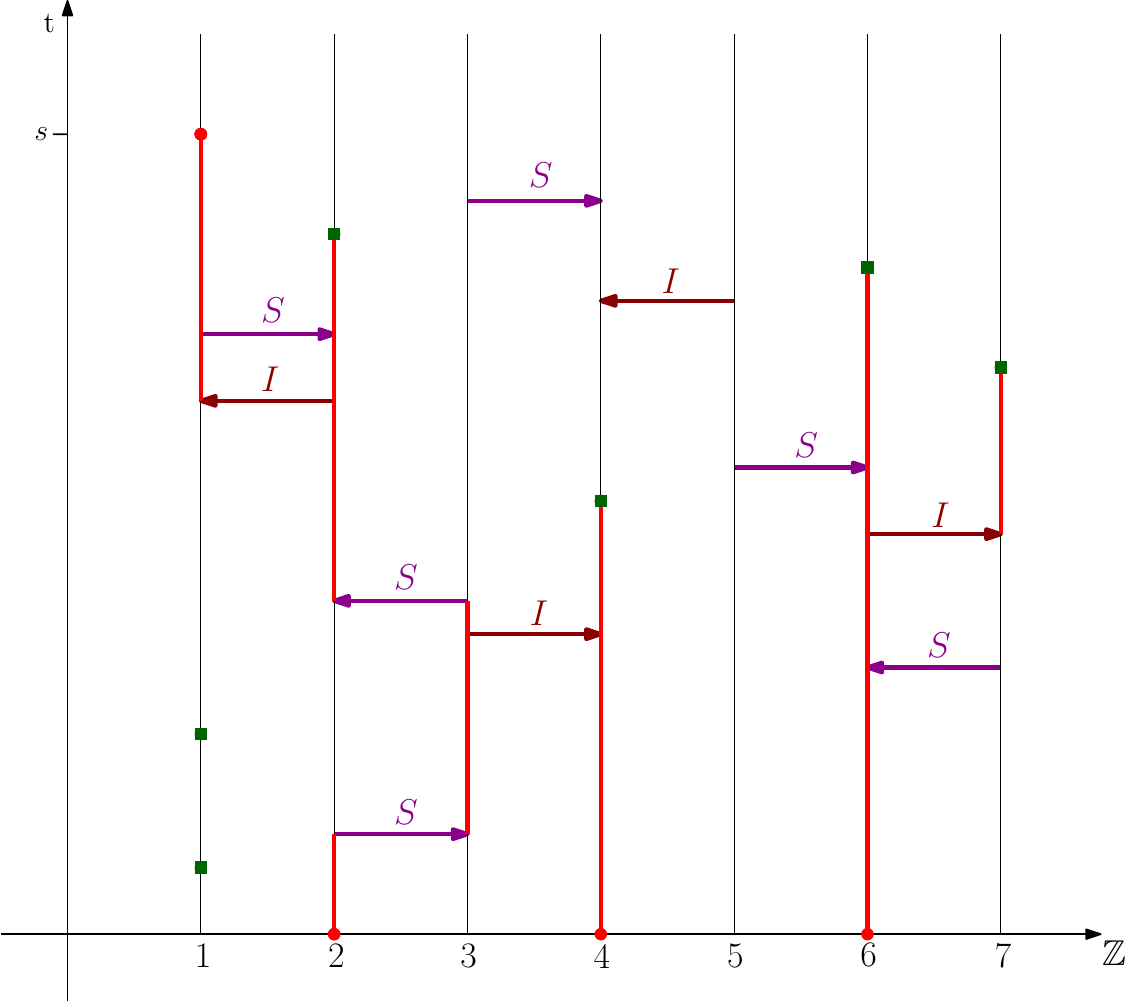}
    \end{minipage}
    \caption{On the left side, the process $(\eta_t^{\CPS})_{t\ge 0}$, and on the right side, the process $(\xi_t^{\CPS})_{t\ge 0}$. Color conventions are the same as in Figure~\ref{Representation_graphique}.}
    \label{fig:Couplage_CPS_et_CP}
\end{figure}

\begin{figure}
    \begin{minipage}{0.49\linewidth}
    \raggedright
    \includegraphics[height=4cm, width=5.5cm]{Representation_graphique_CPS_etape_4.pdf}
    \end{minipage}
    \begin{minipage}{0.49\linewidth}
    \raggedleft
    \includegraphics[height=4cm, width=5.5cm]{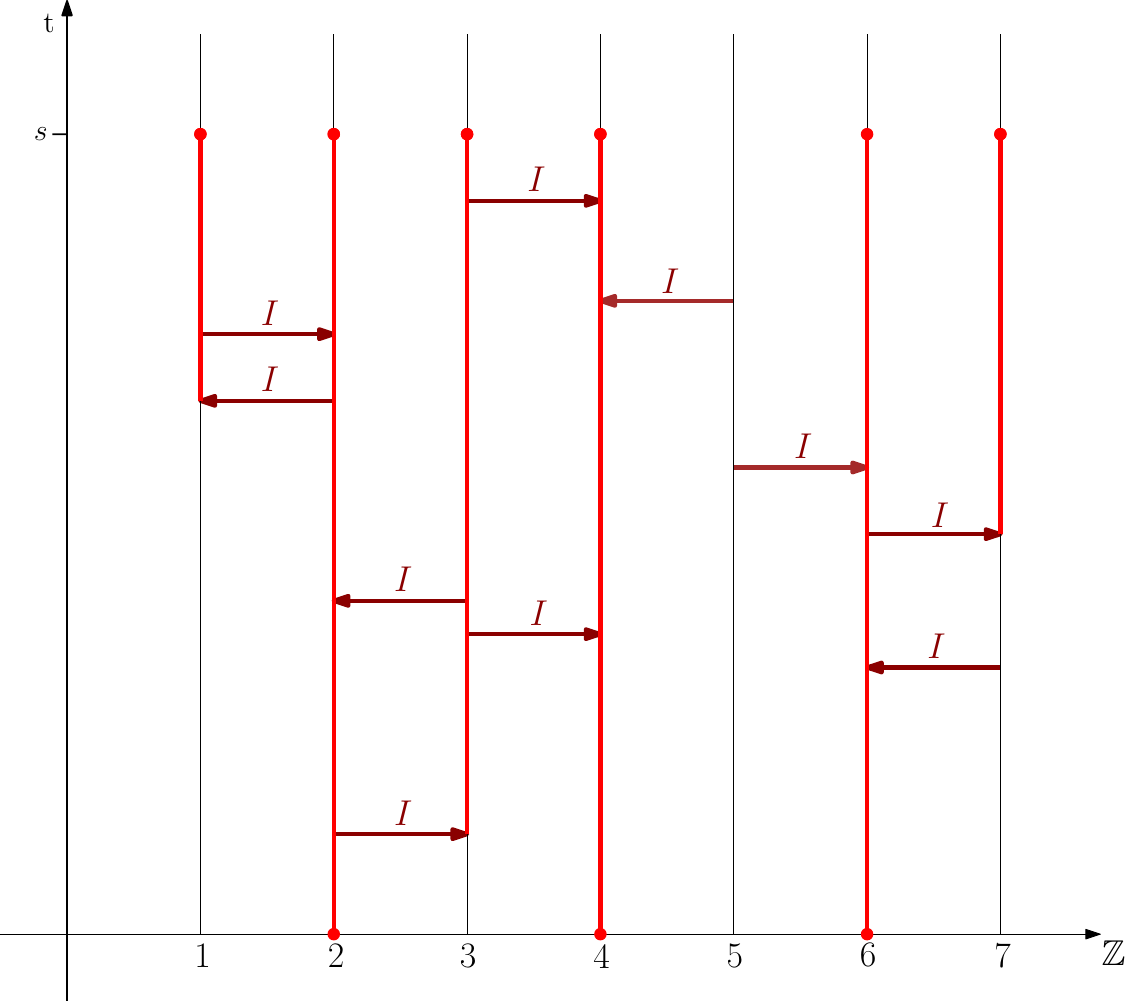}
    \end{minipage}
    \caption{On the left side, the process $(\xi_t^{\CPS})_{t\ge 0}$, and on the right side, the process $(\zeta_t^{\CPS})_{t\ge 0}$. Color conventions are the same as in Figure~\ref{Representation_graphique}.}
    \label{fig:Couplage_CPS_et_RM}
\end{figure}

\begin{lemma} \label{Lemme_sur_le_couplage}
    Let $\lambda,\nu>0$. 
    \begin{enumerate}
        \item There exist processes $(\eta_t^{\RMS})_{t\ge 0}$, $(\xi_t^{\RMS})_{t\ge 0}$, $(\zeta_t^{\RMS})_{t\ge 0}$, all defined on the measurable space $(\Omega,\mathcal F)$, such that:
    \begin{itemize}
        \item the process $(\eta_t^{\RMS})_{t\ge 0}$ is a $\CP(\lambda,2d)$.
        \item the process $(\xi_t^{\RMS})_{t\ge 0}$ is a $\RMS(\lambda,1)$.
        \item the process $(\zeta_t^{\RMS})_{t\ge 0}$ is a $\RM(\lambda+1)$.
        \item if $A\subset B\subset C\subset \Z^d$ and $\eta_0=A$, $\xi_0=B$ and $\zeta_0=C$, then we have, for all $t\ge0$:
        $$\eta_t^{\RMS} \subset \xi_t^{\RMS} \subset \zeta_t^{\RMS}.$$
    \end{itemize}
        \item There exist processes $(\eta_t^{\CPS})_{t\ge 0}$, $(\xi_t^{\CPS})_{t\ge 0}$ and $(\zeta_t^{\CPS})_{t\ge 0}$, all defined on the measurable space $(\Omega,\mathcal F)$, such that:
    \begin{itemize}
        \item the process $(\eta_t^{\CPS})_{t\ge 0}$ is a $\CP(\lambda,2d\nu+1)$.
        \item the process $(\xi_t^{\CPS})_{t\ge 0}$ is a $\CPS(\lambda,1,\nu)$.
        \item the process $(\zeta_t^{\CPS})_{t\ge 0}$ is a $\RM(\lambda+\nu)$.
        \item if $A\subset B\subset C\subset \Z^d$ and $\eta_0=A$, $\xi_0=B$ and $\zeta_0=C$, then we have, for all $t\ge0$:
        $$\eta_t^{\CPS} \subset \xi_t^{\CPS} \subset \zeta_t^{\CPS}.$$
    \end{itemize}
    Moreover, we have 
    $$0<\lambda_c^{\CPS}(d,\nu)\le (2d\nu+1)\lambda_c^{\CP}(d)<+\infty.$$
    \end{enumerate}
\end{lemma}

\begin{proof}
We use the graphical construction: we use notations of Section \ref{Section Construction graphique}. We construct the coupling only for the CPS: the coupling for the RMS is the same, only removing Poisson processes associated with healing, and taking $\nu=1$. Let 
$$\omega=((\omega^{I}_e)_{e\in \E^d},(\omega^{S}_e)_{e\in \E^d},(\omega^{H}_z)_{z\in \Z^d})$$
be a configuration of $\Omega$. We define two configurations $\underline \omega$ and $\overline \omega$ in the following way:

\begin{itemize}
    \item $\underline \omega^I=\omega^I$, $\underline \omega^S=\emptyset$ and for all $z\in \Z^d$, $\underline \omega^H_{z}=\sum\limits_{x\sim z}\omega^S_{(z,x)}$,
    \item $\overline \omega^I=\omega^I+\omega^S$ and $\overline \omega^S=\overline \omega^H=\emptyset$.
\end{itemize}

Note that since each site has $2d$ neighbors, and  
$$\P_{\CPS(\lambda,1,\nu)}=\mathcal{N}_{\lambda}^{\otimes \E^d}\otimes \mathcal{N}_{1}^{\otimes \E^d}\otimes\mathcal{N}_{\nu}^{\otimes \Z^d},$$
then under $\P_{\CPS(\lambda,1,\nu)}$, the law of the coordinate $\omega^H_{z}$ is $\mathcal N_{2d\nu+1}$.

We denote by $\eta=\eta(\underline \omega)$ the process built with the graphical construction constructed from the configuration $\underline \omega$, by $\xi=\xi(\omega)$ the process built with $\omega$, and by $\zeta=\zeta(\overline \omega)$ the process built with $\overline \omega$. It follows directly from the construction that under $\P_{\CPS(\lambda,1,\nu)}$, the process $\eta$ is a $\CP(\lambda,2d\nu+1)$, $\xi$ is a $\CPS(\lambda,1,\nu)$ and $\zeta$ is a
$\RM(\lambda+\nu)$. Moreover we have, for all $A\subset B\subset C\subset \Z^d$ and $t\ge 0$:
$$\eta^{A}_t\subset \xi^{B}_t \subset \zeta^{C}_t.$$
A comparison with a branching process gives $\lambda_c^{\CPS}(d,\nu)>0$: it works in the same way as for the contact process. Finally, since CP of healing rate $2d\nu+1$ is supercritical for all infection rates $\lambda>(2d\nu+1)\lambda_c^{\CP}(d)$, we deduce that:
$$\lambda_c^{\CPS}(d,\nu)\le(2d\nu+1)\lambda_c^{\CP}(d).$$
\end{proof}

The couplings of Lemma \ref{Lemme_sur_le_couplage} are the same as the ones Garet and Marchand used to prove their uniform growth controls in \cite{GaretMarchand2012}: only the parameters of the CP and the RM are different. We do the same restart procedure as they did for the CP. Let us present it for the CPS case: the procedure is the same for the RMS. A CP conditioned to survive has the linear growth properties we want, but the CP in our coupling can die out with positive probability. To counter that, if the CP dies out at time $t$, we restart it at time $t$ from the initial configuration $\delta_x$, with $x$ infected in the RMS, and we repeat this procedure until we start a CP that survives.

Let us describe this procedure formally. We simply write $(\eta_t)_{t\ge 0}$, $(\xi_t)_{t\ge 0}$ and $(\zeta_t)_{t\ge 0}$ for the processes $(\eta_t^{\CPS})_{t\ge 0}$, $(\xi_t^{\CPS})_{t\ge 0}$ and $(\zeta_t^{\CPS})_{t\ge 0}$ of Lemma \ref{Lemme_sur_le_couplage}. We define, for all $x\in \Z^d$, the lifetime of the process $(\eta_t^x)_{t\ge 0}$ starting from the initial configuration $\{x\}$:
$$\tau_x^{\CP}=\inf \{t\geq 0~:~\eta^x_t=\emptyset\}.$$
We define by recurrence a sequence of stopping times $(u_k)_{k\in \N}$ (which will be the restart times) and a sequence of points $(z_k)_{k\in \N}$ (which will be the points from which we start a new CP). We set $u_0=0$, $z_0=0$, and for all $k\geq 0$:
\begin{itemize}
    \item if $u_k<+\infty$ and $\xi_{u_k}\ne\emptyset$, then $u_{k+1}=\tau_{z_k}^{\CP}\circ \theta_{u_k}+u_k$,
    \item if $u_k=+\infty$ or $\xi_{u_k}=\emptyset$, then $u_{k+1}=+\infty$.
    \item if $u_{k+1}<+\infty$ and $\xi_{u_{k+1}}\ne\emptyset$, then $z_{k+1}$ is the smallest point of $\xi_{u_{k+1}}$ for the lexicographical order,
    \item if $u_{k+1}=+\infty$ or $\xi_{u_{k+1}}=\emptyset$, then $z_{k+1}=+\infty$.
\end{itemize}
We denote by $(\bar\eta_t)_{t\geq 0}$ the process obtained with the following restart procedure. Informally, the process $\bar\eta$ is the contact process $\eta$ from time $u_0=0$ to time $u_1$. At time $u_1$, we have $\eta_{u_1}=\emptyset$: the contact process $\eta$ dies out. At the same time, the process $\bar\eta$ restarts: starting from time $u_1$, it is the contact process $(\eta_t^{z_1})_{t\ge u_1}$. After that, when this translated contact process dies out, at time $u_2$, we restart it from a new point $z_2$ in the same manner. We say that the procedure \emph{stops} if 
$$K:=\inf \{k\in \N~:~u_{k+1}=+\infty\text{ or }z_k=+\infty\}$$
is finite. Note that, by construction, we cannot have simultaneously $u_k<+\infty$, $u_{k+1}=+\infty$ and $z_k=+\infty$: the procedure stops if the CPS dies out, or if the last restarting CP survives. In this latter case, the last restarting process has the law of a CP conditioned to survive, and the CPS survives. Note also that in the case of the RMS, the process $(\xi_t)_{t\ge 0}$ cannot die, so if the procedure stops, then it is because the last restarting CP survives.

More formally, the process $\bar\eta$ is defined as follows:
$$\forall0\leq k <K, ~ \forall u_k\leq t < u_{k+1},~\bar\eta_t=\eta_{t-u_k}\circ\theta_{u_k}\circ T_{z_k}$$
and
$$\forall t\geq u_K,~ \bar\eta_t=\eta_{t-u_K}\circ\theta_{u_K}\circ T_{z_K}.$$ 
By construction of the coupling of Lemma \ref{Lemme_sur_le_couplage}, we have directly the following result: for all $A\subset B \subset C \subset \Z^d$ and $t\ge 0$, 
\begin{align}
    \bar\eta_t^A\subset \xi_t^B. \label{Le_PC_qui_restart_minore}
\end{align}

We define the time of first infection of the site $x\in \Z^d$ in the process $(\bar \eta_t)_{t\ge 0}$:
$$t_{\overline{\CP}}(x)=\inf \{t\geq 0:x\in \bar \eta_t\},$$
With the couplings of Lemma \ref{Lemme_sur_le_couplage} and \eqref{Le_PC_qui_restart_minore}, we prove \eqref{Le modèle est permanent}, \eqref{AML}, \eqref{SC} and \eqref{ALL} for both the RMS and the CPS, for all infection rates $\lambda$ such that, respectively, the contact processes $(\eta_t^{\RMS})_{t\ge 0}$ and $(\eta_t^{\CPS})_{t\ge 0}$ are supercritical.
\begin{proposition} \label{TFA:Controles_demandes_pour_le_RMS_et_le_CPS}
\begin{enumerate}
    \item Let $\lambda>2d\lambda_c^{\CP}(d)$ and $(\xi_t)_{t\ge 0}$ be a $\RMS(\lambda,1)$. There exist some constants $A,B,M_1,M_2>0$ such that, for all $t\ge 0$ and $x\in \Z^d$:
\begin{align*}
\P(\exists y\in \Z^d,~t(y)\le t \text{ and } ||y||\ge M_1t)&\le A\exp(-Bt),\\
\P(t(x)>M_2||x||+t)&\le A\exp(-Bt).
\end{align*}
    \item  Let $\nu>0$, $\lambda>(2d\nu+1)\lambda_c^{\CP}(d)$ and $(\xi_t)_{t\ge 0}$ be a $\CPS(\lambda,1,\nu)$. There exist some constants $A,B,M_1,M_2>0$ such that, for all $t\ge 0$ and $x\in \Z^d$:
\begin{align*}
\P(\exists y\in \Z^d,~t(y)\le t  \text{ and } ||y||\ge M_1t)&\le A\exp(-Bt),\\
\P(t<\tau<+\infty)&\le A\exp(-Bt),\\
\P(t(x)>M||x||+t)&\le A\exp(-Bt).
\end{align*}
\end{enumerate}

\end{proposition}
\begin{proof}
The process $(\bar \eta_t)_{t\ge 0}$ is exactly the same as the one Garet and Marchand defined in \cite{GaretMarchand2012} to prove \eqref{AML}, \eqref{SC} and \eqref{ALL} for the contact process in random environment, only using properties of the process $(\bar \eta_t)_{t\ge 0}$ and Lemmas \ref{Lemme_sur_le_couplage} and \eqref{Le_PC_qui_restart_minore}. Therefore, their proof can be adapted verbatim to obtain the growth controls \eqref{AML}, \eqref{SC} and \eqref{ALL} for the RMS (resp. the CPS), for all $\lambda$ such that the contact process $(\eta_t^{\RMS})_{t\ge 0}$ (resp. $(\eta_t^{\CPS})_{t\ge 0}$) is supercritical. Since we bound from below the RMS by a $\CP(\lambda,2d)$ (resp. the CPS by a $\CP(\lambda,2d\nu+1)$), then we have the growth controls for all $\lambda>2d\lambda_c^{\CP}(d)$ for the $\RMS(\lambda,1)$ (resp. $\lambda>(2d\nu+1)\lambda_c^{\CP}(d)$ for the $\CPS(\lambda,1,\nu)$). 
\end{proof}

\begin{proof}[Proof of Theorem \ref{Théorème TFA d>1}]
By Lemma \ref{Définition du modèle: le processus est un growth model} and Proposition \ref{TFA:Controles_demandes_pour_le_RMS_et_le_CPS}, the $\RMS(\lambda,1)$ is in the class $\mathcal C\cap \mathcal C_L$ of Deshayes and Siest \cite{DeshayesSiest2024} for all $\lambda>2d\lambda_c^{\CP}(d)$. For the $\CPS(\lambda,1,\nu)$, we have all the growth controls for $\nu>0$ and 
\begin{align*}
    \lambda>\max((2d\nu+1)\lambda_c^{\CP}(d),\lambda_c^{\CPS}(d,\nu))>(2d\nu+1)\lambda_c^{\CP}(d),
\end{align*}
by Lemma \ref{Lemme_sur_le_couplage}. Therefore, we can apply Theorem 1 of Deshayes and Siest \cite{DeshayesSiest2024} in both cases to obtain Theorem \ref{Théorème TFA d>1}.
\end{proof}
\section{Proof of Proposition \ref{Théorème Proposition sur densité de sites infectés}} \label{Section_Preuve_de_la_Proposition_sur_la_densité}

\subsection{Proof of an isoperimetric inequality}

Let $\lambda>0$ and $(\xi_t)_{t\ge 0}$ be a $\RMS(\lambda,1)$. For all $x\in \Z^d$, we denote by $v(x)$ the set of neighboring sites of $x$. We define, for all $1\leq i\leq 2d$,
$$\Fr_i(\xi_t):=\{x\in \Z^d ~:~ x\notin \xi_t \text{ and } \Card\left(v(x)\cap \xi_t\right)=i\},$$
the set of healthy sites that have exactly $i$ neighbors infected at time $t$, and
$$\Fr(\xi_t):=\bigcup_{1\leq i\leq 2d} \Fr_i(\xi_t)$$
the set of healthy sites that have at least one infected neighbor, which is also the set of healthy sites that can be infected at time $t$. A healthy site that has exactly $i$ infected neighbors is infected at rate $i\lambda$. At time $t$, a new infection arises at rate $I(\xi_t)\lambda$, where
$$I(\xi_t):=\sum_{1\leq i \leq 2d}i \Card[\Fr_i(\xi_t)].$$ 
We state the following discrete isoperimetric inequality, proved for example by Hamamuki in \cite{HamamukiIsopérimétrie}:

\begin{lemma}[Discrete isoperimetric inequality] \label{Inegalite_isoperimetrique}
    Let $A\subset \Z^d$ be a non-empty finite connected subset of $\Z^d$. There exists a constant $C=C_d>0$ such that:
    $$\Card(\Fr (A))\geq C \Card(A)^{1-\frac{1}{d}}.$$
\end{lemma}

This discrete isoperimetric inequality allows us to bound from below $I(\xi_t)$, by bounding from below the cardinality of the set of healthy sites that can be infected at time~$t$. 

\begin{proposition} \label{Inegalite_isoperimetrique_sur_notre_modele}
There exists a constant $C=C_d>0$ such that, for all $t>0$, we have
$$I(\xi_t)\geq C \Card(\xi_t)^{1-\frac{1}{d}}.$$
\end{proposition}

\begin{proof}
We denote by $\mathcal{C}$ the set of connected components of $\xi_t$. In order to bound from below the quantity $I(\xi_t)$, we use Lemma \ref{Inegalite_isoperimetrique} on each element of $\mathcal{C}$: there exists $C=C_d>0$ such that
\begin{align}\sum\limits_{A\in \mathcal{C}} \Card(\Fr (A))\geq C \sum_{A\in \mathcal{C}} \Card(A)^{1-\frac{1}{d}}. \label{Inegalite_isoperimetrique_Minoration_de_Card(Fr(A))}
\end{align}
But we know that $I(\xi_t)\geq \sum\limits_{A\in \mathcal{C}} \Card(\Fr (A))$. With these inequalities, and by \eqref{Inegalite_isoperimetrique_Minoration_de_Card(Fr(A))}, we deduce that:
\begin{align}
I(\xi_t)=\sum_{1\leq i \leq 2d}i \Fr_i(\xi_t)&\geq \sum\limits_{A\in \mathcal{C}} \Card(\Fr(A)) \notag\\
&\geq C \sum_{A\in \mathcal{C}} \Card(A)^{1-\frac{1}{d}} \label{Line 2}\\
&\geq C\left(\sum\limits_{A\in \mathcal{C}} \Card(A)\right)^{1-\frac{1}{d}} \geq C \Card(\xi_t)^{1-\frac{1}{d}}.\label{Line 3}
\end{align}
To go from \eqref{Line 2} to \eqref{Line 3} we used that, for all $d\geq 2$, $n\in \N^*$ and $a_1,...,a_n\geq 0$, we have
$$\sum\limits_{i=1}^n a_i^{1-\frac{1}{d}}\geq \left(\sum\limits_{i=1}^n a_i\right)^{1-\frac{1}{d}}.$$
\end{proof}

\subsection{Proof of Proposition \ref{Théorème Proposition sur densité de sites infectés}}

We recall that we defined a constant $C$ in Proposition \ref{Inegalite_isoperimetrique_sur_notre_modele}. Let $(Y_t)_{t\geq 0}$ be a birth process with birth rate $q_i=C i^{1-\frac{1}{d}} \lambda$, $i\in \N$, and initial configuration $Y_0=1$: it is a Markov process on $\N$ which jumps from $i$ to $i+1$ at rate $q_i$. For more information about this process, see for example \cite{Norris1997}. We have $\Card(\xi_0)=1$, and conditionally on $\xi_t$ and $\Card(\xi_t)=i$, the next infection in $\xi_t$ happens at rate $I(\xi_t)\lambda$. Moreover, by Proposition \ref{Inegalite_isoperimetrique_sur_notre_modele}, we have $I(\xi_t)\lambda\geq q_i.$ Therefore, we can make a coupling between the process $(Y_t)_{t\geq 0}$ and the process $(\Card(\xi_t))_{t\geq 0}$ for which the birth process is dominated by the process $(\Card(\xi_t))_{t\geq 0}$. Therefore, we only need to show that $\P[Y_t \ge M t^d]\ge 1-\left(\frac{C\lambda}{t}\right)$.

We set, for all $k\in \N$ and $t\geq 0$,
$$F(k)=\sum\limits_{i=1}^{k-1} i^{\frac{1}{d}-1}\quad \text{and} \quad X_t=F(Y_t)-C\lambda t.$$

\noindent\underline{Step 1.} Let us show that the process $(X_t)_{t\geq 0}$ is a martingale. The process $(Y_t)_{t\geq 0}$ is a birth process with birth rate $(q_i)_{i\in \N}$, so by denoting by $\L$ its generator we have, for all continuous functions $f:\R \to \R$ and $n\in \N$, 
\begin{align*}
    \L f(n)&=q_n\left(f(n+1)-f(n)\right).
\end{align*}
In particular, we have: 
\begin{align*}
    \L F(n)&=q_n\left(F(n+1)-F(n)\right)=C \lambda.
\end{align*}
Applying for example Theorem 6.2 of \cite{LeGall2016}, we obtain that the process $(X_t)_{t\geq 0}$ is a martingale. We deduce in particular that we have, for all $t\geq 0$, $\E(X_t)=\E(X_0)=0$ and $\E[F(Y_t)]=C \lambda t$.

\noindent\underline{Step 2.} Let us show that we have, for all $t\geq 0$, 
$$\Var(X_t)\leq C \lambda t.$$
We have
\begin{align*}
    X_t^2\le F(Y_t)^2
    \le \left(\sum\limits_{i=1}^{Y_t-1} i^{\frac{1}{d}-1}\right)^2 
    \le \left(\int_1^{Y_t} x^{\frac{1}{d}-1}dx\right)^2
    \le d^2Y_t^{\frac{2}{d}}
    \le d^2Y_t,
\end{align*}
since $d\ge 2$. Since the process $(Y_t)_{t\ge 0}$ is stochastically dominated by a Yule process of birth rate $C\lambda$ (that is, a birth process with birth rates $iC\lambda$), then $Y_t$ is integrable, and so $X_t^2$ is integrable too. Therefore, $X_{t}$ has a finite variance. Now let $t,s>0$. We recall that the conditional variance of $X_{t+s}$ with respect to $X_t$ is defined by:
    \begin{align*}
        \Var(X_{t+s}|X_t)&=\E[(X_{t+s}-\E[X_{t+s}|X_t])^2|X_t]=\E[(X_{t+s}-X_t)^2|X_t],
    \end{align*}
    the second equality coming from the fact that here, $(X_t)_{t\ge 0}$ is a martingale, by Step $1$. By the law of total variance, we have
    \begin{align}
        \Var(X_{t+s})&=\E[\Var(X_{t+s}|X_t)]+\Var(\E[X_{t+s}|X_t]) =\E[\Var(X_{t+s}|X_t)]+\Var(X_t), \label{Processus de naissance Formule de la variance totale}
    \end{align}
    the second equality coming from the fact that $(X_t)_{t\ge 0}$ is a martingale. Since $(Y_u)_{u\geq 0}$ is a birth process, we know that (see for example \cite{Norris1997}):
    \begin{align*}
        \P[Y_{t+s}=Y_t+1|Y_t]=q_{Y_t}s+o(s) \quad \text{and}\quad \P[Y_{t+s}=Y_t|Y_t]=1-q_{Y_t}s+o(s).
    \end{align*}
    We deduce that:
    \begin{align*}
        \P[X_{t+s}=X_t+Y_t^{\frac{1}{d}-1}-C\lambda s|X_t]&=q_{Y_t}s+o(s)\\
        \text{and}\quad\P[X_{t+s}=X_t-C\lambda s|X_t]&=1-q_{Y_t}s+o(s).
    \end{align*}
    Therefore we have, using \eqref{Processus de naissance Formule de la variance totale}:
    \begin{align*}
    \frac{\Var(X_{t+s})-\Var(X_t)}{s}=\frac{\E[\Var(X_{t+s}|X_t)]}{s}&=\frac{\E[\E[(X_{t+s}-X_t)^2|X_t]]}{s} \\ &=\E[Y_t^{\frac{2}{d}-2}q_{Y_t}]+o(1)\\
    &=\E[Y_t^{\frac{2}{d}-2}CY_t^{1-\frac{1}{d}}\lambda] +o(1)\\
    &=\E[Y_t^{\frac{1}{d}-1}C\lambda]+o(1)\le C \lambda +o(1),
    \end{align*}
the inequality coming from the fact that $Y_t\in \N^*$ and that $\frac{1}{d}-1<0$. Since the domain of the generator is the entire set of continuous functions, the map $t\mapsto\E[X_t^2]$ is differentiable, and we have $\frac{d}{dt} \Var(X_t)\leq C \lambda$. Therefore: 
$$\Var(X_t)\le C \lambda t.$$

\noindent\underline{Step 3.} We extend the function $F$ by linear interpolation to an increasing and bijective function, still denoted by $F$, from $[1,+\infty)$ to $[0,+\infty)$. Let us show that, for all $t$ large enough,
\begin{align}
F^{-1}(dt^{\frac{1}{d}})\ge t. \label{Preuve Propositon 2.4: équivalent}
\end{align}
For all $k\in \N^*$, we have
\begin{align*}
    \int_0^{k-1} x^{\frac{1}{d}-1}dx \le \sum\limits_{i=1}^{k-1} i^{\frac{1}{d}-1} \le \int_1^{k} x^{\frac{1}{d}-1}dx,
\end{align*}
and so:
\begin{align*}
    d\left(k-1\right)^{\frac{1}{d}}\le F(k) \le dk^{\frac{1}{d}}-1, \quad\text{therefore}\quad F(k)\le dk^{\frac{1}{d}}\le F(k+1).
\end{align*}
Since $F^{-1}$ is increasing, we have:
\begin{align*}
   k&\le F^{-1}(dk^{\frac{1}{d}}).
\end{align*}
Finally, since $F$ is obtained by linear interpolation from $F_{|\N^*}$, we have \eqref{Preuve Propositon 2.4: équivalent}.

\noindent\underline{Step 4.} Now we prove that:
$$\P(\Card(\xi_t)\ge Mt^d)\ge 1-\frac{4}{C\lambda t}.$$
Using that $F$ is increasing and \eqref{Preuve Propositon 2.4: équivalent}, we have for all $t$ large enough:
\begin{align*}
    1-\P\left[|F(Y_t)-C \lambda t|\ge \frac{C \lambda}{2}t \right]&\le1-\P\left[C \lambda t-F(Y_t)\ge \frac{C \lambda}{2}t \right]\\
    &\le\P\left[C \lambda t-F(Y_t)\le \frac{C \lambda}{2}t \right]\\
    &\le\P\left[F(Y_t)\ge \frac{C \lambda}{2}t  \right]\\
    &\le\P\left[Y_t \ge F^{-1}\left(\frac{C \lambda}{2}t\right)\right]\\
    &\le \P\left[Y_t \ge F^{-1}\left(d\left[\left(\frac{C \lambda}{2d}\right)^d t^d\right]^{\frac{1}{d}}\right)\right]\\
    &\le \P\left[Y_t \ge \left(\frac{C \lambda}{2d}\right)^d t^d\right]\le \P[Y_t \ge B t^d],
\end{align*} 
where $B=\left(\frac{C \lambda}{2d}\right)^d$. Moreover, by the Tchebychev inequality we have, for all $t>0$,
\begin{align*}
    \P\left(|F(Y_t)-C \lambda t|\ge  \frac{C \lambda}{2}t \right)\le \frac{4}{C \lambda t},
\end{align*}
so we have finally:
\begin{align*}
    \P(\Card(\xi_t)\ge Bt^d)&\ge \P[Y_t \ge B t^d]\ge 1-\frac{4}{C \lambda t}\ge 1-\frac{A}{\lambda t}, \quad\text{with }A=\frac{4}{C}.
\end{align*}

\section{Proof of Theorem \ref{Théorème Fixation RMS}}


We define a family of particles indexed by $\Z^d$, whose positions evolve with time:
\begin{itemize}
    \item at any time, there is exactly one particle per site,
    \item the particles move according to stirring times: two neighboring particles exchange their positions at a stirring time if and only if, at this time, they are located on two sites in different states.
\end{itemize}
Note that the particle $a \in \Z^d$ is not necessarily located on site $a$. More formally, for every $a \in \Z^d$, we define the process $(P_t^a)_{t \geq 0}$, called the \emph{trajectory of particle $a$}, using the graphical construction of the model (see Section \ref{Section Construction graphique}). The initial positions of the particles $(P_0^a)_{a \in \Z^d}$ define a partition of $\Z^d$. Then, for all $t > 0$, $P_t^a$ is obtained by following the path in the graphical construction starting from $(P_0^a, 0)$, and moving along the temporal lines with the following constraints:
\begin{enumerate} 
\item The path does not pass through any infection arrow, 
\item The path follows a stirring arrow if and only if there is exactly one infected site incident to the arrow. \label{Une particule se déplace avec contrainte} 
\end{enumerate}
We define a process $(\Tilde{\xi}_t)_{t \geq 0}$ that carries the information about the state and position of each particle over time. For all $t \geq 0$, we set $\Tilde{\xi}_t=(\xi_t\circ P_t,P_t)$, where $P_t$ is the map
$$\begin{array}{ccc}
\Z^d & \longrightarrow & \Z^d \\ a & \longmapsto & P_t^a 
\end{array}.$$ 
By construction, we immediately have the following lemma: 
\begin{lemma} \label{Lemme Propriétés des particules} 
\begin{enumerate} 
\item At every time $t > 0$, the positions $(P_t^a)_{a \in \Z^d}$ of the particles at time $t$ form a partition of $\Z^d$. 
\item The process $(\Tilde{\xi}_t)_{t \geq 0}$ satisfies the strong Markov property. \label{Le processus des particules vérifie propriété Markov forte} 
\item Once a particle $a \in \Z^d$ is infected, it remains infected forever: 
\[
\text{if } \Tilde{\xi}_t(a) = (1, P_t^a), \text{ then: } \forall s \geq t,~\Tilde{\xi}_s(a) = (1, P_s^a). \label{Particule infectée un jour, infectée toujours} 
\] 
\item A site $x \in \Z^d$ is healthy at time $t$ if and only if a healthy particle is located at $x$ at time $t$: 
$$x \notin \xi_t \iff \exists a \in \Z^d,~\Tilde{\xi}_t(a) = (0, x).$$ 
\label{Origine saine à l'instant t si une particule est venue soigner} 
\end{enumerate} 
\end{lemma} 
Point \eqref{Origine saine à l'instant t si une particule est venue soigner} of Lemma \ref{Lemme Propriétés des particules} justifies the introduction of these processes: we will obtain information about the state of sites by studying the trajectories of the particles. 


\subsection{Proof of the first point of Theorem \ref{Théorème Fixation RMS}}

\noindent\underline{Step 1.} \label{Fixation Preuve Dim 1-2 étape 1}
We show that almost surely, for all $x\in \Z^d$, only finitely many particles reach the site $x$ while remaining healthy, that is:
\begin{equation} \label{Fixation Nombre fini de billes}
    \#\{a\in \Z^d ~:~ \iota_a^x<+\infty\}<+\infty \quad \text{a.s.},
\end{equation}
where
$$\iota_a^x=\inf\{t\geq 0 ~:~ \Tilde{\xi}_t(a)=(0,x)\}.$$ 
We will prove \eqref{Fixation Nombre fini de billes} for $x=0$: the proof is similar for any $x$. Let $\eta$ denote the configuration $(\mu,\mathrm{Id})$, where $\mu$ is the initial configuration of the process (non-empty a.s.), and $\mathrm{Id}$ is the configuration where each particle $a$ starts at site $a$. For a fixed particle $a\in\Z^d\setminus \{0\}$, we show:
\begin{equation*}
    \P_\eta (\iota_a^0<+\infty)\leq \left(\frac{1}{1+\lambda}\right)^{|a|_1}, 
\end{equation*}
where $|a|_1=\sum_{i=1}^{d}|a_i|$ and $\P_\eta$ is the conditional probability starting from the initial configuration $\eta$. Fix $a \in \Z^d\setminus \{0\}$ and $n:=|a|_1$. We define $n+1$ stopping times $(u_k)_{0\leq k \leq n}$ as follows: set $u_0=0$, and for all $1\leq k \leq n$, $u_k$ is the first time when particle $a$ is at a distance $n-k$ from the origin while remaining healthy:
$$u_{k}=\inf\left\{t\geq 0: \exists y\in \Z^d,~|y|_1=n-k\text{ and }\Tilde{\xi}_t(a)=(0,y) \right\}.$$
Note that for $0\leq k_1<k_2 \leq n$, we have $u_{k_1}\le u_{k_2}$, and $u_n=\iota_a^0$. Since the process $(\Tilde{\xi}_t)_{t\geq 0}$ satisfies the strong Markov property (point \eqref{Le processus des particules vérifie propriété Markov forte} of Lemma \ref{Lemme Propriétés des particules}), we have:
\begin{align*}
    \P_\eta[\iota_a^0<+\infty]&= \P_\eta[u_{n-1}<+\infty,~\exists t>u_{n-1},~\Tilde{\xi}_t(a)=(0,0)]\\
    &= \P_\eta[u_{n-1}<+\infty]\P_{\Tilde{\xi}_{u_{n-1}}}[u_n<+\infty]\\
    &= \P_\eta(u_1<+\infty) \prod_{i=1}^{n-1} \P_{\Tilde{\xi}_{u_{n-i}}}[u_{n-i+1}<+\infty].
\end{align*}
Starting from the initial configuration $\Tilde{\xi}_{u_{n-i}}$, $P_0^a\in \mathcal{N}_i=\{x\in \Z^d: |x|_1=i\}$. Furthermore, when particle $a$ reaches a distance $i-1$ from the origin for the first time, there could have been an infection instead of the exchange that allowed the particle to move (recall how particles move \eqref{Une particule se déplace avec contrainte}). The probability of infection before a stirring time on an edge is $\frac{\lambda}{\lambda+1}\in (0,1)$, and such an infection would imply that $u_{n-i+1}=+\infty$. Using similar arguments for $\P_\eta(u_1<+\infty)$, we get:

\begin{align*}
    \P_\eta[\iota_a^0<+\infty]\le \left(\frac{1}{1+\lambda}\right)^{n}. 
\end{align*}
Hence, we have:
\begin{align*}
    \sum_{a\in \Z^d}\P(\iota_a^0<+\infty)&= \sum_{n\in \N} \Card(\{a\in \Z^d: |a|_1=n\})\left( \frac{1}{1+\lambda}\right)^n<+\infty,
\end{align*}
and by the Borel-Cantelli lemma, we deduce \eqref{Fixation Nombre fini de billes} for $x=0$.

\noindent\underline{Step 2.} We prove the third point of Theorem \ref{Théorème Fixation RMS}. We have shown that for all $x\in \Z^d$, the set
$$H^x=\{a\in\Z^d: \iota^x_a<+\infty\}$$
is finite a.s. (see \eqref{Fixation Nombre fini de billes}). Moreover, each time a healthy particle moves, it does so by swapping with an infected particle. For each such swap, the healthy particle could have been infected by the infected particle with a rate of at least $\lambda>0$, and these potential infections are independent. Thus, a healthy particle moves only a finite number of times a.s. This implies that almost surely, after some time $T>0$, all healthy particles in $H^x$ stop moving. Consequently, after time $T>0$, the site $x$ is either occupied by the same healthy particle forever or by an infected particle forever.

Finally, two neighboring sites cannot fix in different states: once fixed, the infected site would infect the healthy site at rate $\lambda>0$. Therefore, all sites fix in the same state a.s. This proves the third point of Theorem \ref{Théorème Fixation RMS}.


\subsection{Proof of the second point of Theorem \ref{Théorème Fixation RMS}} \label{Fixation Preuve Dim 1-2 étape 2} 

Here, the dimension $d$ is one or two. We define, for $y\in \xi_0$, the process $(M_t^y)_{t\geq 0}$, which gives, for all $t\ge 0$, the site reached by following the graphical construction path starting from $(y,0)$, taking all stirring arrows (with no constraints on the states of the incident sites), and only those, up to time $t$. Note that $M_t^y$ is always on an infected site, as $M_0^y\in \xi_0$. The process $(M_t^y)_{t\geq 0}$ is a simple random walk in continuous time with rate $2d$. Since this random walk is recurrent, it implies that every site $z\in \Z^d$ is occupied by an infected particle infinitely often a.s. However, as shown in the previous proof, all sites fix in the same state from a certain time onward a.s., implying that all sites fix in the infected state a.s.

\subsection{Proof of the third point of Theorem \ref{Théorème Fixation RMS}}

Here, $d\ge 1$ and $\lambda>2d\lambda_c(d)$. We aim to show that
\begin{equation}
    \P(\forall y\in \Z^d,~\{t\geq 0: y\in \xi_t\} \text{ is unbounded})=1, \label{Fixation Ensemble des temps où x est infecté est non borné}
\end{equation}
which was shown in the previous proof using a simple random walk, but here the latter is transient because $d\geq 3$. To prove \eqref{Fixation Ensemble des temps où x est infecté est non borné}, we use the asymptotic shape theorem proven for $\lambda>2d\lambda_c(d)$ (see Theorem \ref{Théorème TFA d>1}). Let $y\in \Z^d$, $t_0>0$, and $z$ be the smallest infected site of $\xi_{t_0}$ for the lexicographical order. Let us prove that there exists $t>t_0$ such that $y$ is infected at time $t$. We define the process $(\overline{\xi}_{t}^{t_0})_{t\ge 0}$, which is a RMS starting from the initial configuration $\overline{\xi}^{t_0}_{0}=\{z\}$, constructed using the same environment as the RMS $(\xi_{t})_{t\ge 0}$, but shifted in time by $\theta_{t_0}$: for all $t\ge 0$, we set
$$\overline{\xi}^{t_0}_t=\xi_t\circ\theta_{t_0} \circ T_{z}.$$
It is clear that this process is a lower bound for the initial RMS $(\xi_{t})_{t\ge t_0}$ starting from time $t_0$: for all $t\ge 0$, we have $\overline{\xi}^{t_0}_t\subset \xi_{t_0+t}$. Furthermore, since $\lambda>2d\lambda_c(d)$, we can apply Theorem \ref{Théorème TFA d>1} to the process $(\overline{\xi}^{t_0}_{t})_{t\ge 0}$, which implies in particular that almost surely, there exists $T>0$ such that $y\in \overline{\xi}^{t_0}_{T}$. Consequently, $y\in \xi_{T+t_0}$ a.s. This proves \eqref{Fixation Ensemble des temps où x est infecté est non borné}.

\section{Proof of Theorem \ref{Théorème Survie forte et faible sur Td}}

Let $d\ge 1$ and $T_d$ be a homogeneous tree of degree $d+1$, that is, an infinite tree for which all vertices (which we still call \emph{sites}) have degree $d+1$. We still write $y\sim x$ if $x\in T_d$ and $y\in T_d$ are neighbors.
The tree $T_1$ is simply $\Z$, so we consider $d\ge 2$. We fix a site $r\in T^d$, and we define a map $l_r:T_d\to \Z$ such that:
\begin{itemize}
    \item We have $l_r(r)=0$.
    \item For all $x\in T_d$, we have $l_r(y)=l_r(x)-1$ for exactly one site $y\sim x$, and $l_r(y)=l_r(x)+1$ for exactly $d$ sites $y\sim x$.
\end{itemize}
We simply denote by $l$ this map. We say that $l(x)$ is the \emph{level} of $x$ in the tree $T_d$, which can also be seen as the generation number of site $x$. See Figure~\ref{Figure: Arbre homogene} for an example with $d=2$.

Let $A\subset T_d$. For all $A\subset T_d$, we define the RMS and the CPS with initial configuration $A$ on $T_d$, denoted by $(\xi_t)_{t\ge 0}$, in the same way as on $\Z^d$. 

\begin{figure}[H]
    \centering
    \includegraphics[height=8cm, width=8cm]{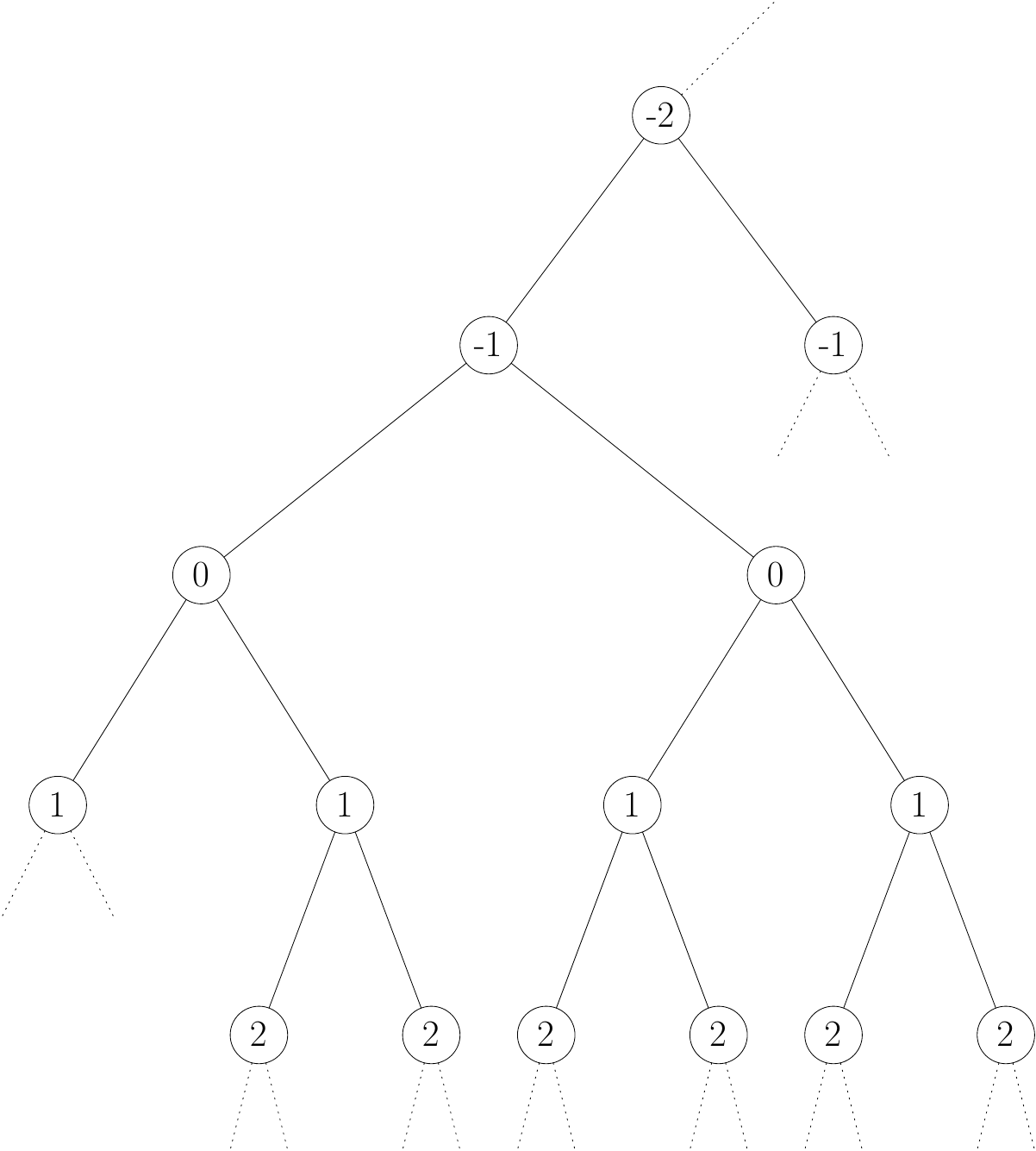}
    \caption{A part of the homogeneous tree $T_2$. On each site, we write its level in the tree.}
    \label{Figure: Arbre homogene}
\end{figure}

\subsection{Proof of Theorem \ref{Théorème Survie forte et faible sur Td} for the RMS}

Clearly, no matter the infection rate $\lambda$, the $\RMS(\lambda,1)$ survives weakly, since there is no healing in this model. We adapt Liggett's proof \cite{LiggettSIS} of the existence of an interval of infection rates $\lambda$ such that the $\CP(\lambda,1)$ survives weakly but not strongly on $T_d$. It amounts to prove that, for all $\nu>0$, there exists $\lambda_s^{\RMS}(d,\nu)>0$ such that, for all $\lambda<\lambda_s^{\RMS}$, the $\CPS(\lambda,1,\nu)$ does not survive strongly.

We define, for all $\rho\in (0,1)$ and $A\subset T_d$ finite, the quantity
$$w_\rho(A):=\sum_{x\in A} \rho^{l(x)}.$$
\noindent\underline{Step 1:} We start by proving that the process $(w_\rho(A_t))_{t\ge 0}$ is a positive supermartingale, for $\rho=\frac{1}{\sqrt{d}}$ and all $0<\lambda\le \frac{d+1}{2\sqrt{d}}-1$. To do that, we show that we have, for all finite $A\subset T_d$:
\begin{align}
    \frac{d}{dt}\E^A\left[w_\rho(\xi_t) \right]_{|t=0}\le 0. \label{Survie faible: dérivée en 0 négative}
\end{align}
By Theorem B3 of Liggett's book \cite{LiggettSIS}, we have:
\begin{align}
    \frac{d}{dt}\E^A\left[w_\rho(\xi_t) \right]_{|t=0}&=\sum_{x\in A} \left[ \lambda \sum_{\substack{y\sim x \\ y\notin A}} \rho^{l(y)}+1\times \sum_{\substack{y\sim x \\ y\notin A}} \left(\rho^{l(y)}-\rho^{l(x)}\right)\right] \notag\\
    &=\sum_{x\in A} \rho^{l(x)} \left[\sum_{\substack{y\sim x \\ y\notin A \\ l(y)=l(x)+1}} \left[(\lambda+1)\rho-1\right]+ \sum_{\substack{y\sim x \\ y\notin A \\ l(y)=l(x)-1}} \left[\frac{\lambda+1}{\rho}-1\right]\right]. \label{Dérivée de l'espérance de w_rho étape 1}
\end{align}

Note that we always have $\frac{\lambda+1}{\rho}-1\ge 0$, but for some values of $\lambda$ and $\rho$, we could have $(\lambda+1)\rho-1\le 0$.

We can split the sum \eqref{Dérivée de l'espérance de w_rho étape 1} into sums on the connected components of $A$, and prove that each of these sums is negative. Therefore, we can suppose without loss of generality that the set $A$ is connected. In this case, there exists a unique site $x_0\in A$ such that $x_0$'s parent is not in $A$. Moreover, this site $x_0$ is the only site in $A$ which verifies
$$l(x_0)=\min\{l(y):y\in A\}.$$
For all $n\in \Z$, we denote by:
\begin{itemize}
    \item $A_n$ the set of sites that are in $A$ and have level $n$;
    \item $a_n$ the quantity $\Card(A_n)$.
\end{itemize}
Note that we have $w_\rho(A)=\sum_{n\ge l(x_0)} a_n \rho^n$, and the quantity $da_n-a_{n+1}$ corresponds to the number of sites of level $n+1$ that are not in $A$, but have a parent in $A$. In other words, it is the number of sites of level $n+1$ that are not in $A_{n+1}$, but can be infected by someone in $A$. With these notation, \eqref{Dérivée de l'espérance de w_rho étape 1} can be rewritten as:
\begin{align*}
    &\sum_{n\ge l(x_0)} \rho^{n} (da_n-a_{n+1}) \left[(\lambda+1)\rho-1\right]+\left[\frac{\lambda+1}{\rho}-1\right]\rho^{l(x_0)} \notag\\
    &=\left[dw_\rho(A)-\left(\frac{1}{\rho}w_\rho(A)-a_{l(x_0)}\rho^{l(x_0)-1}\right)\right] \left[(\lambda+1)\rho-1\right]+  \left[\frac{\lambda+1}{\rho}-1\right]\rho^{l(x_0)} \notag\\
    &=\left[\left(d-\frac{1}{\rho}\right)w_\rho(A)+\rho^{l(x_0)-1}\right]\left[(\lambda+1)\rho-1\right]+  \left[\frac{\lambda+1}{\rho}-1\right]\rho^{l(x_0)}.
\end{align*}
For $\rho=\frac{1}{\sqrt{d}}$ and $\lambda<\frac{d+1}{2\sqrt{d}}-1$, we have:
\begin{align}
    &\left[\left(d-\sqrt{d}\right)w_\frac{1}{\sqrt{d}}(A)+\left(\frac{1}{\sqrt{d}}\right)^{l(x_0)-1}\right]\left[\frac{\lambda+1}{\sqrt{d}}-1\right]+\left[\sqrt{d}(\lambda+1)-1\right]\left(\frac{1}{\sqrt{d}}\right)^{l(x_0)}\notag\\
    &\le \left(\frac{1}{\sqrt{d}}\right)^{l(x_0)} \left[d\left[\frac{\lambda+1}{\sqrt{d}}-1\right]+\sqrt{d}(\lambda+1)-1\right] \label{Majoration avec w_rho 1 RMS}\\
    &\le \left(\frac{1}{\sqrt{d}}\right)^{l(x_0)} \left[2\sqrt{d}(\lambda+1)-(d+1)\right] \label{Majoration avec w_rho 2 RMS}<0,
\end{align}
where:
\begin{itemize}
    \item we obtain \eqref{Majoration avec w_rho 1 RMS} by using that $w_\frac{1}{\sqrt{d}}(C)\ge a_{l(x_0)}\left(\frac{1}{\sqrt{d}}\right)^{l(x_0)}=\left(\frac{1}{\sqrt{d}}\right)^{l(x_0)}$ and that $\frac{\lambda+1}{\sqrt{d}}-1<0$,
    \item we have \eqref{Majoration avec w_rho 2 RMS} because of the choice of $\lambda$ and $d\ge 2$,
\end{itemize} 
and so we proved \eqref{Survie faible: dérivée en 0 négative}. 

Now let us prove that \eqref{Survie faible: dérivée en 0 négative} implies that $(w_\rho(\xi_t))_{t\ge 0}$ is a supermartingale. We denote by $(\mathcal{F}_t)_{t\ge 0}$ the natural filtration associated to the process $(\xi_t)_{t\ge 0}$, and by $\mathcal{L}=\mathcal{L}^{\RMS}_{\lambda,1}$ the generator of the process $(\xi_t)_{t\ge 0}$. By Theorem 6.2 of \cite{LeGall2016}, the process
$$\left(w_\rho(\xi_t)-\int_0^t \mathcal{L}w_\rho(\xi_s)ds\right)_{t\ge 0}$$
is a $\mathcal{F}_t$-martingale. Since we have, for all finite $A\subset \Z^d$,
\begin{align*}
\mathcal{L}w_\rho(\eta)&=\frac{d}{dt}\E^A[w_\rho(\xi_t)]_{|t=0}<0,
\end{align*}
then we have, for all $t,s\ge 0$,
\begin{align*}
    \E[w_\rho(\xi_{t+s})|\mathcal{F}_t]&=\E\left[w_\rho(\xi_{t+s})-\int_0^{t+s} \mathcal{L}w_\rho(\xi_{u})du|\mathcal{F}_t\right]+\E\left[\int_0^{t+s} \mathcal{L}w_\rho(\xi_u)du|\mathcal{F}_t\right]\\
    &=w_\rho(\xi_{t})-\int_0^{t} \mathcal{L}w_\rho(\xi_{u})du+\int_0^{t} \mathcal{L}w_\rho(\xi_{u})du+\E\left[\int_t^{t+s} \mathcal{L}w_\rho(\xi_{u})du|\mathcal{F}_t\right]\\
    &\le w_\rho(\xi_{t}),
\end{align*}
and $(w_\rho(\xi_t))_{t\ge 0}$ is a positive supermartingale.\\

\noindent\underline{Step 2:} Let us show that all sites are fixing on a state after a long enough time. Since the process $(w_\rho(\xi_t))_{t\ge 0}$ is a positive supermartingale for $\rho=\frac{1}{\sqrt{d}}$ and $0<\lambda<\frac{d+1}{2\sqrt{d}}-1$ (by Step $1$), then for these parameters, $(w_\rho(\xi_t))_{t\ge 0}$ converges a.s.

Let $n\in \Z$ and $x\in T_d$ be a site such that $|l(x)|\le n$. If an infection or an exchange happens on $x$ at time $t$, then the quantity $w_\rho(\xi_t)$ increases or decreases of at least a constant $\alpha_n>0$, which does not depend on the time $t$. Since the process $(w_\rho(\xi_t))_{t\ge 0}$ converges a.s., we deduce that each site $x$ such that $|l(x)|\le n$ is fixing on infected or healthy after a long enough time. Moreover, two neighboring sites cannot fix on a different state, since they would have positive probability to interact. Therefore, they all fix on the same state. \\

\noindent\underline{Step 3:} Let us show that the sites fix on healthy after a long enough time a.s. Suppose by contradiction that a site fixes on infected after a long enough time a.s. Then all the sites fix on infected after a long enough time a.s. 
This implies that:
$$\sum_{x\in T^d} \rho^{l(x)}\le \lim_{t\to +\infty} w_{\rho} (\xi_t)\quad \text{a.s.},$$
which is absurd since $\sum_{x\in T^d} \rho^{l(x)}\ge\sum_{n\in \N} \left(\frac{1}{\rho}\right)^{n} =+\infty$.\\

\noindent\underline{Step 4:} Finally, let us show \eqref{Bornes paramètre critique survie forte RMS}. The bound from below has just been proved. For the bound from above, we couple the $\RMS(\lambda,1)$ with a $\CP(\lambda,d+1)$ in the same way as we did on $\Z^d$ in Lemma \ref{Lemme_sur_le_couplage}. The healing rate is $d+1$ here, since each site has $d+1$ neighbors. Moreover, denoting by $\lambda_s^{\CP}(d)$ the critical parameter for the strong survival of the contact process of infection rate $\lambda$ and of healing rate $1$ on $T_d$, then Theorem 4.65 of Liggett's book \cite{LiggettIPS} gives
$$\lambda_s^{\CP}(d)\le \frac{1}{\sqrt{d}-1},\quad \text{and so }\lambda_s^{\RMS}(d)\le \frac{d+1}{\sqrt{d}-1}.$$

\subsection{Proof of Theorem \ref{Théorème Survie forte et faible sur Td} for the CPS}

We follow the same steps as in the preceding proof. 

\noindent\underline{Step 1:} Let us show that, for $\rho=\frac{1}{\sqrt{d}}$ and all
$0<\lambda<\nu\left(\frac{d+1}{2\sqrt{d}}-1\right)$, we have:
\begin{align}
    \frac{d}{dt}\E^A\left[w_\rho(\xi_t) \right]_{|t=0}\le 0. \label{Survie faible: dérivée en 0 négative CPS}
\end{align}
In the same manner as for the RMS, and with the same notations, we have:
\begin{align*}
    \frac{d}{dt}\E^A\left[w_\rho(\xi_t) \right]_{|t=0}&=\sum_{x\in A} \left[\lambda \sum_{\substack{y\sim x \\ y\notin A}} \rho^{l(y)}+\nu \sum_{\substack{y\sim x \\ y\notin A}} \left(\rho^{l(y)}-\rho^{l(x)}\right)- \rho^{l(x)}\right]\\
    &=\nu\sum_{x\in A}  \left[\frac{\lambda}{\nu} \sum_{\substack{y\sim x \\ y\notin A}} \rho^{l(y)}+1\times \sum_{\substack{y\sim x \\ y\notin A}} \left(\rho^{l(y)}-\rho^{l(x)}\right)\right]- \sum_{x\in A}\rho^{l(x)}.
\end{align*}
Remember that in the proof for the RMS, we showed that:
$$\nu\sum_{x\in A} \left[\frac{\lambda}{\nu} \sum_{\substack{y\sim x \\ y\notin A}} \rho^{l(y)}+1\times \sum_{\substack{y\sim x \\ y\notin A}} \left(\rho^{l(y)}-\rho^{l(x)}\right)\right]<0$$
for $\rho=\frac{1}{\sqrt{d}}$ and $\frac{\lambda}{\nu}<\frac{d+1}{2\sqrt{d}}-1$. Then for $\rho=\frac{1}{\sqrt{d}}$ and $\lambda<\nu\left[\frac{d+1}{2\sqrt{d}}-1\right]$, we have \eqref{Survie faible: dérivée en 0 négative CPS}.\\

\noindent\underline{Step 2 and 3:} They are exactly the same as for the RMS. \\

\noindent\underline{Step 4:} Now let us bound from above the critical parameter $\lambda_s^{\CPS}(d,\nu)$. Similarly to the RMS case, we couple a $\CPS(\lambda,1,\nu)$ with a $\CP(\lambda,(d+1)\nu+1)$. As for the RMS, we use that
$\lambda_s^{\CP}(d)\le \frac{1}{\sqrt{d}-1}$, which gives: 
$$\lambda_s^{\CPS}(d,\nu)\le \frac{(d+1)\nu+1}{\sqrt{d}-1}.$$
Therefore we have 
\begin{align*}
        \nu\left(\frac{d+1}{2\sqrt{d}}-1\right)\le\lambda_s^{\CPS}(d,\nu)\le\frac{(d+1)\nu+1}{\sqrt{d}-1}.
\end{align*}
On the other hand, if we denote by $\lambda_w^{\CP}(d)$ the critical parameter for the weak survival of the $\CP(\lambda,1)$, then Theorem 4.1 of Liggett's book \cite{LiggettIPS} gives
$$\lambda_w^{\CP}(d)\le \frac{1}{d-1}.$$
So the coupling of a $\CPS(\lambda,1,\nu)$ and a $\CP(\lambda,(d+1)\nu+1)$ also gives:
$$\lambda_w^{\CPS}(d,\nu)\le \frac{(d+1)\nu+1}{d-1}.$$
Therefore, our bounds give $\lambda_w^{\CPS}(d,\nu)<\lambda_s^{\CPS}(d,\nu)$ if and only if:
\begin{align}
    \nu\left(\frac{d+1}{2\sqrt{d}}-1\right)-\frac{(d+1)\nu+1}{d-1}> 0 \label{Condition arbre}
\end{align}
A quick study of the map $f:x\mapsto x^2-4x^{\frac{3}{2}}-1$ shows that $f$ is negative on $(0,16]$ and positive on $[17,+\infty)$. Therefore:
\begin{itemize}
    \item If $d\le 16$, we have 
    $$\eqref{Condition arbre}\iff \nu<\frac{2\sqrt{d}}{d^2-4d^{\frac{3}{2}}-1}<0.$$
    \item If $d\ge 17$, we have 
    $$\eqref{Condition arbre}\iff \nu>\frac{2\sqrt{d}}{d^2-4d^{\frac{3}{2}}-1}>0.$$
\end{itemize}
It implies that $W$ is non-empty if and only if $d\ge 17$.

\section{Discussion} \label{Section discussion}

\subsection{Asymptotic shape theorem for low infection rates}

In dimension $d>1$, Durrett and Griffeath \cite{DurrettGriffeath1982} proved the asymptotic shape theorem for the contact process, for infection rates $\lambda$ large enough, using embedded one-dimensional contact processes. One could want to proceed in a similar way for the RMS. The problem is that the natural coupling between a RMS and an embedded one-dimensional RMS is not monotone, in the sense that we do not necessarily have $\xi^1_t\subset \xi^d_t$, where $(\xi^1_t)_{t\ge 0}$ (resp. $(\xi^d_t)_{t\ge 0}$) is an embedded one-dimensional RMS (resp. a RMS in dimension $d$): see Figure~\ref{Couplage entre RMS}. The problem is the same for the CPS: the stirring dynamics break the monotonicity.

\begin{figure}
    \centering
    \includegraphics[scale=0.6]{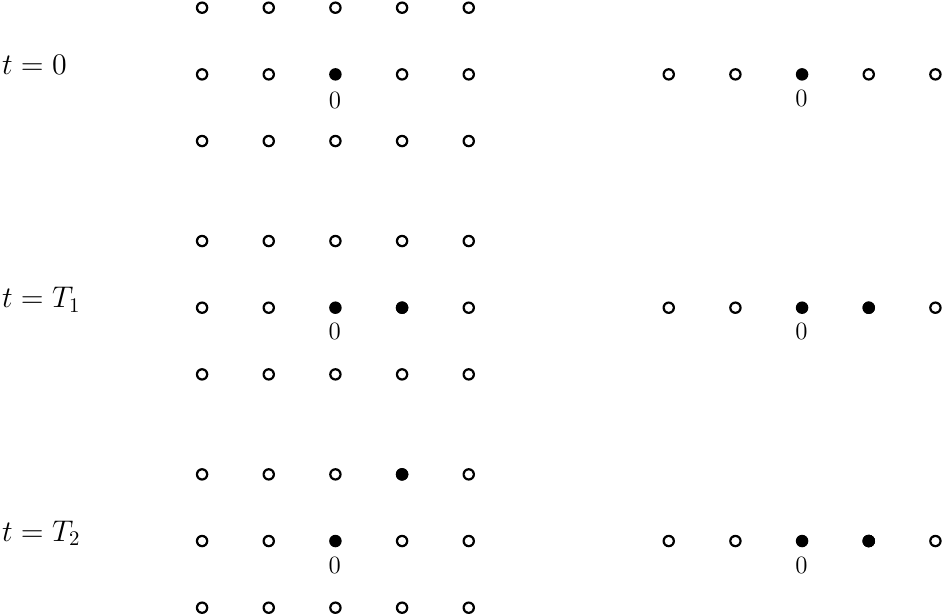}
    \caption{A two-dimensional contact process (left side) and an embedded one-dimensional contact process on $\Z\times \{0\}$. Infected sites are black, healthy sites are white. A time $t=0$, both processes start from the initial state $\{0\}$. At time $t=T_1$, the first interaction occurs: site $(0,0)$ infects site $(1,0)$, and so site $0$ infects site $1$ in the one-dimensional process. Time $t=T_2$ is a stirring time: site $(1,0)$ and site $(1,1)$ exchange their states, and nothing happens in the one-dimensional process. At this time, we have $\xi_{T_2}^1\not\subset \xi_{T_2}^d$.}
    \label{Couplage entre RMS}
\end{figure}

Another way to deal with the case of low infection rates could be to use Proposition \ref{Théorème Proposition sur densité de sites infectés} to prove that the growth is at least linear \eqref{ALL}. For $x=0$, the idea is that at some time $t$, there are around $t^d$ infected sites (Proposition \ref{Théorème Proposition sur densité de sites infectés}), relatively close to the origin (coupling with RM, Lemma \ref{Lemme_sur_le_couplage}), with high probability. At this time, imagine that there is a particle placed on each infected site, each of them following the time lines of the graphical construction. These particles are all making a simple random walk at rate $2d\nu$, where $\nu$ is the stirring rate. If these walks were independent, we could prove that the origin is infected at linear speed, since there are numerous infected particles doing a random walk near the site. But we did not manage to overcome the strong dependencies between these random walks.

\subsection{Bezuidenhout and Grimmett's block construction}

Bezuidenhout and Grimmett \cite{BezuidenhoutGrimmett1990} made a smart block construction in order to compare the contact process to an oriented percolation. It is possible to deduce from it many properties, such as the fact that the contact process dies out at criticality, or that the critical parameters for weak and strong survival are equal. It can also be used to prove \eqref{SC} for the contact process. One could want to use a similar construction adapted to the RMS and the CPS to prove \eqref{SC} and \eqref{ALL} for these models, or to prove that the CPS dies out at criticality. But this construction relies on the FKG inequality to prove that some events have high enough probabilities. By Theorem 2.14 of \cite{LiggettIPS}, the RMS (resp. the CPS) on a finite subset $S\subset \Z^d$ satisfies the FKG inequality if and only if its generator, rewritten in the form
$$\L f(\xi)=\sum_{\eta\in \{0,1\}^{S^d}} \rho(\xi,\eta)[f(\eta)-f(\xi)],$$
with $\rho(\xi,\eta)$ the rate at which the process goes from the configuration $\xi$ to $\eta$, satisfies:
$$\rho(\eta,\xi)>0 \text{ implies that }\eta\le \xi\text{ or }\eta\ge \xi.$$
But with the stirring, an infected site and a healthy neighbor on a configuration $\xi$ can exchange their states at a positive rate, and the configuration $\eta$ obtained does not check $\eta\le \xi$ nor $\eta\ge \xi$. Therefore, the FKG inequality appears to be false in general for both the RMS and the CPS on $\Z^d$, and it is therefore not possible to directly adapt Bezuidenhout and Grimmett's construction for our models. 
\printbibliography

\end{document}